\documentclass{amsart}
\usepackage{amsmath}
\usepackage{amssymb}
\usepackage[all]{xy}
\usepackage[percent]{overpic}
\usepackage{graphicx}
\usepackage{xcolor}
\usepackage{enumitem}
\usepackage{hyperref}
\hypersetup{
    colorlinks=true,       
    linkcolor=blue,          
    citecolor=blue,        
    filecolor=blue,      
    urlcolor=blue           
}

    \newtheorem{Lem}{Lemma}[section]
    \newtheorem{Lem-Def}[Lem]{Lemma-Definition}
    \newtheorem{Prop}[Lem]{Proposition}

    \newtheorem{Thm}[Lem]{Theorem}  
    \newtheorem{Cor}[Lem]{Corollary}
    \newtheorem{Conj}[Lem]{Conjecture}
		
\theoremstyle{definition}
\font\smallsc=cmcsc10
\font\smallsl=cmsl10

    \newtheorem{Def}[Lem]{Definition}
    \newtheorem{Exa}[Lem]{Example}
    \newtheorem{Rem}[Lem]{Remark}

\newcommand{\rank}{\text{rank}}
\newcommand{\<}{\lesssim}
\renewcommand{\>}{\gtrsim}
\newcommand{\E}{\mathcal E}

\newcommand{\Spec}{\text{Spec}\,}

\newcommand{\A}{\mathcal A}
\newcommand{\B}{\mathcal B}

\newcommand{\M}{\mathcal M}

\newcommand{\T}{\mathcal T}
\newcommand{\K}{\mathcal K}

\renewcommand{\L}{\mathcal L}
\renewcommand{\O}{\mathcal O}

\newcommand{\V}{\mathcal V}

\newcommand{\Y}{\mathcal Y}

\newcommand{\R}{\mathcal R}
\newcommand{\col}{\colon}

\newcommand{\ol}{\overline}

\newcommand{\lra}{\longrightarrow}

\newcommand{\Aut}{\text{Aut}}
\newcommand{\cone}{\text{cone}}

\begin{document}

\title{Enriched curves and their tropical counterpart}

\author{Alex Abreu and Marco Pacini}


\begin{abstract}
\noindent  In her Ph.D. thesis, Main\`o introduced the notion of enriched structure on stable curves and constructed their moduli space. In this paper we give a tropical notion of enriched structure on tropical curves and construct a moduli space parametrizing these objects. Moreover, we use this construction to give a toric description of the scheme parametrizing enriched structures on a fixed stable curve. 
\end{abstract}

\maketitle

\tableofcontents


\section{Introduction}

 This paper is devoted to the study of enriched curves. The main motivation behind the definition of enriched stable curve is given by the theory of limit linear series. One of the techniques to study linear series on smooth curves is to degenerate the curve to a singular one and then, by understanding the properties of the limits, recover properties of the initial smooth curve.
  However, a typical phenomenon which occurs when one considers the degeneration of an object from a smooth curve to a singular one, is that the limit often depends on the chosen degeneration. As a consequence, the theory of limit linear series was originally only developed in \cite{EH} for curves of compact type, for which the limit does not depend on the degeneration.\par
	To take in account this phenomenon, Main\`o  in \cite{maino} introduced the notion of enriched stable curve. An enriched stable curve is, essentially, the datum of a stable curve together with a direction in which to smooth the curve.  She also constructed the moduli space $\E_g$ parametrizing enriched stable curves of genus $g$. The space $\E_g$ is not proper, and it would be interesting to find a modular compactification of $\E_g$. Nevertheless, there is a natural compactification $\overline{\E}_g$ of $\E_g$ obtained via blowups of the moduli space $\overline{\M}_g$ of Deligne-Mumford stable curves of genus $g$.\par
	We point out that the notion of enriched curve also appears in \cite{EM}, with a slightly different approach, to study limits of Weierstrass points for nodal curves with two components. Also, the theory of limit linear series saw important progress in the last  two decades, for instance see \cite{Es}, \cite{O}, \cite{Rizzo} and \cite{O1}.\par\medskip

		This paper was originally motivated by the attempt to give an answer to the following problems.\par
\begin{enumerate}
\item Find a modular description of the natural compactification $\overline{\E}_g$, or, at least, of the fiber of $\overline{\E}_g\to\overline{\M}_g$ over a point in $\overline{\M}_g$.
\item Give the tropical equivalent of an enriched curve of genus $g$ and construct the moduli space $E_g^{trop}$ parametrizing these objects.
\item Establish the relationship between $E_g^{trop}$ and $M_g^{trop}$, the moduli space parametrizing tropical curves of genus $g$.
\item Use the tropical setting to give a toric description of the fiber of $\overline{\E}_g\to\overline{\M}_g$ over a point in $\overline{\M}_g$.
\end{enumerate}

	  Recently, tropical geometry has proven to be a promising tool for solving problems in algebraic geometry and many analogies have been made between algebraic and tropical geometry. For instance, see \cite{M}, \cite{BN}, \cite{MZ}, \cite{Caporaso}, \cite{BMV}, \cite{CDPR} and \cite{AC}.  Hence, a solution to the last three problems may lead to a solution of the first one.\par
		In this paper we just deal with tropical geometry in dimension one, i.e., with tropical curves. A tropical curve is a vertex weighted (multi-)graph $(\Gamma,w)$, where $w\col V(\Gamma)\to \mathbb{Z}_{\geq0}$ is the weight function, together with a function $l\col E(\Gamma) \to \mathbb{R}_{>0}$, where $E(\Gamma)$ is the set of edges of $\Gamma$, called the length function of $(\Gamma,w,l)$.\par
		In \cite{BMV}, Brannetti, Melo and Viviani showed that the moduli space $M_g^{trop}$ of tropical curves of genus $g$ and the moduli space $A_g^{trop}$ of tropical principally polarized abelian varieties of dimension $g$ can be constructed as stacky fans, and defined the tropical Torelli map between these two spaces. A stacky fan is, set-theoretically, a disjoint union of cells, where each cell is an open rational cone modulo the action of some finite group (see Section \ref{sub:stacky} for a precise definition). The cells of $M_g^{trop}$ parametrize the tropical curves $(\Gamma,w,l)$ with fixed underlying vertex weighted graph $(\Gamma,w)$, hence they can be identified with $\mathbb{R}^{|E(\Gamma)|}_{>0}/\Aut(\Gamma,w)$.\par
	  Since $\overline{\E}_g$ is a blowup of $\overline{\M}_g$, by analogy with the theory of toric blowups, it is expected that the stacky fan $E_g^{trop}$ cited in Problem (2) is a ``refinement'' of $M_g^{trop}$, i.e., there is a set-theoretically bijection $E_g^{trop}\to M_g^{trop}$, such that the image of every cell of $E_g^{trop}$ is contained in a cell of $M_g^{trop}$. Moreover, since a cell $\mathbb{R}^{|E(\Gamma)|}_{>0}/\Aut(\Gamma,w)$ of $M_g^{trop}$ parametrizes tropical curves with fixed underlying graph by varying the length function, a refinement of such a cell is just a disjoint union 
\[
\mathbb{R}^{E(\Gamma)}_{>0}/\Aut(\Gamma,w)=\coprod_{p\in P} C_p,
\] 
where each $C_p$ can be obtained simply by imposing linear inequalities on the lengths of the edges.
		Hence, a tropical enriched curve is just a tropical curve together with a preorder on its set of edges (see Definition \ref{def:enrichedgraph} and Section \ref{sec:moduli}). 
		
		Our first result, contained in Theorem \ref{thm:maintrop} and which solves Problems (2) and (3), can be stated as follows.
	
\begin{Thm}
\label{thm:A}
There exists a stacky fan $E_g^{trop}$ parametrizing tropical enriched curves of genus $g$, whose cells parametrize tropical enriched curves with fixed underlying vertex weighted graph and preorder. Moreover there exists a natural forgetful map of stacky fans $E_g^{trop}\to M_g^{trop}$ which is a set-theoretical bijection.
\end{Thm}		
		
		In fact, the refinement of $\mathbb{R}^{|E(\Gamma)|}_{>0}/\Aut(\Gamma,w)$ comes from a refinement of $\mathbb{R}^{|E(\Gamma)|}_{>0}$, see Proposition \ref{prop:union}, and such a refinement gives rise to a fan $\Sigma_\Gamma$, see Definition \ref{def:fan}. Using this fan we can construct a toric variety for any given stable curve $Y$ with dual graph $\Gamma$. In this toric variety we consider a distinguished invariant toric subvariety, which we call $\E_Y$, see Section \ref{sec:enriched} for the precise construction. 
		
		We can state our next result, which answers to Problem (4) and which is contained in Theorem \ref{thm:mainEX} and Corollary \ref{cor:maino}, as follows.
				
\begin{Thm}
\label{thm:B}
If $Y$ is a stable curve with no nontrivial automorphisms, then the toric variety $\E_Y$ is isomorphic to the fiber of $\overline{\E}_g\to\overline{\M}_g$ over the point $[Y]\in\overline{\M}_g$.
\end{Thm}
		
		 This result suggests that our definition of tropical enriched curve should be the correct tropical equivalent of an enriched curve, as defined by Main\`o. In \cite{ACP}, Abramovich, Caporaso and Payne exhibited a geometrically meaningful connection between $\overline{\M}_g$ and $\overline{M}_g^{trop}$, where $\overline{M}_g^{trop}$ is the compactification of $M_g^{trop}$ constructed by Caporaso in \cite[Theorem 3.30]{Caporaso} by means of extended tropical curves of genus $g$. Building on this work, several similar results for various moduli spaces recently appeared, for instance see \cite{CMR}, \cite{CHMR} and \cite{U}. Following the work of Caporaso, it should be possible to construct a compactification $\ol{E}^{trop}_g$ of $E_g^{trop}$, and hence investigate similar connections between $\overline{\E}_g$ and $\ol{E}_g^{trop}$.\par
		
		Finally, in Theorem \ref{thm:maininclusion} and Corollary \ref{cor:P}, we give a more concrete description of $\E_Y$ as follows.

\begin{Thm}
\label{thm:C}
Given a stable curve $Y$ with $n$ nodes and dual graph $\Gamma$, there exist subspaces $V_i\subset\mathbb{C}^{n}$, which can be explicitly described in terms of $\Gamma$, such that the natural rational morphism 
\[
\mathbb{P}(\mathbb{C}^{n})\dashrightarrow\prod \mathbb{P}(\mathbb{C}^n/V_i).
\]
has image isomorphic to $\E_Y$.  Moreover, one can find explicit equations of such an image. If $\Gamma$ is biconnected, then the induced map $\mathbb{P}(\mathbb{C}^n)\dashrightarrow \E_Y$ is birational.
\end{Thm}
		
	We  believe that the above result can be used to solve the weak part of Problem (1). Indeed, the images in the above theorem have been modularly described by Li in \cite{Li}, although this description is not explicitly related to enriched curves. It is also interesting to note that the equations of $\E_Y$ are similar to the ones found by Batyrev and Blume in \cite[Corollary 1.16]{BM} describing toric varieties associated to root systems. \par
	This paper is structured as follows. In Section \ref{sec:preliminaries} we recall the following basic tools used throughout the paper. First, we start with the definition of preorder, fixing some notation, and introducing some terminology in graph theory. Second, we review some of the theory of toric varieties. Then we recall some of the results about enriched structures contained in \cite{maino}. Finally, we define stacky fans following \cite{BMV}. In Section \ref{sec:enrichedgraphs} we define enriched graphs and prove several results that lay the foundations for the following two sections. In Section \ref{sec:moduli} we define enriched tropical curves and prove Theorem \ref{thm:A}. In Section \ref{sec:enriched}, we define the toric variety $\E_Y$ and prove Theorems \ref{thm:B} and \ref{thm:C}. Finally, in Section \ref{sec:final}, we make some considerations about a future work related to Problem (1).

\section{Preliminaries}
\label{sec:preliminaries}

\subsection{Preorders}
\label{sec:preorder}
  A \emph{preorder} on a set $E$ is a binary relation $\<$ satisfying the following properties:
\begin{enumerate}
\item[(1)] $a\< a$, for all $a\in E$.
\item[(2)] if $a\< b$ and $b\< c$ then $a\< c$ for all $a,b,c\in E$.
\end{enumerate}

If a preorder also satisfies the property:
\begin{enumerate}
\item[(3)] if $a\< b$ and $b\< a$ then $a=b$,
\end{enumerate}
then it is called a \emph{partial order}, and denoted by $\leq$. Moreover if every two elements on $E$ are comparable then the partial order is called a \emph{total order}.\par
  Since in this paper we will vary the preorder on a set $E$, we will often denote a preorder by $p$ and its relation by $\<_p$. Two elements $a$ and $b$ are called \emph{incomparable} by $p$ if neither $a\<_p b$ nor $b\<_p a$. Given a subset $S$ of the preordered set $(E,p)$, we define the preorder $p|_S$ as the restriction of $p$ to $S$.\par
	Given a preordered set $(E,\<_p)$, an \emph{upper set} (with respect to $p$) is a subset $S\subset E$ such that if $a\in S$ and $a\<_p b$ then $b\in S$. A \emph{lower set} (with respect to $p$) is a subset $S\subset E$ such that if $b\in S$ and $a\<_p b$ then $a \in S$. Clearly, the complement of an upper set is a lower set and vice versa. A preorder $p$ on $E$ induces a topology where the closed sets are the upper sets. Hence, we say that an upper set is \emph{irreducible} if it is so in such a topology.\par
		Given a preorder $p$ on a set $E$, we define the equivalence relation $\sim_p$ on $E$ as $a\sim_p b$ if and only if $a\<_p b$ and $b\<_p a$. We will write $a<_p b$ if $a\<_p b$ and $a\nsim_p b$. Let $E_p$ denote the quotient $E/\sim_p$, and $[a]_p$ denote the equivalence class of an element $a$ of $E$. Clearly $p$ induces a partial order $\leq_p$ on the set $E_p$.\par
	Given a finite partially ordered set $(E,p)$, we define the \emph{Hasse diagram} of $(E,p)$ as the directed graph whose vertices are the elements of $E$ and with a direct edge from $a$ to $b$ if and only if $a<_p b$ and there is no element $c\in E$ such that $a<_p c<_p b$.\par
	Throughout the paper, for a given set $E$ and a ring $R$ we will define the free $R$-module
\begin{equation}\label{eq:RE}
R^E:=\bigoplus_{e\in E} R\cdot e.
\end{equation}
We will often write an element $\oplus x_e\cdot e\in R^E$ as $(x_e)_{e\in E}$. We call $(x_e)$ the \emph{coordinates} of this element.

\subsection{Graphs}
\label{subsec:graphs}
  Let $\Gamma$ be a graph. We denote by $E(\Gamma)$ the set of edges of $\Gamma$ and by $V(\Gamma)$ its set of vertices. If $\Gamma$ is connected, a vertex of $\Gamma$ is called a \emph{separating vertex} if there is a loop attached to it (and at least one more edge attached to it) or the graph becomes nonconnected after its removal; if $\Gamma$ is nonconnected a \emph{separating vertex} of $\Gamma$ is a separating vertex of one of its connected components.  
  The \emph{valence} of a vertex $v$ is the number of edges incident to $v$. We say that $\Gamma$ is \emph{$k$-regular} if all its vertices have valence $k$. If $\Gamma$ is $2$-regular and connected then it is called a \emph{circular graph}. A \emph{cycle} of $\Gamma$ is a circular subgraph of $\Gamma$. Given a set of vertices $V\subset V(\Gamma)$ we denote by $\Gamma(V)$ the subgraph of $\Gamma$ whose set of vertices is $V$ and whose edges are the edges of $\Gamma$ connecting two vertices (possibly the same vertex, in the case of a loop) in $V$.\par
  We call $\Gamma$ \emph{biconnected} if it is connected and has no separating vertices. A \emph{biconnected component} of $\Gamma$ is a maximal biconnected subgraph of $\Gamma$. Equivalently a biconnected component of $\Gamma$ can be defined as (the graph induced by) a maximal set of edges such that any two edges lie on a cycle of $\Gamma$ (see \cite[page 558]{CLRS}). We note that each loop together with its vertex is a biconnected component. Any connected graph decomposes uniquely into a tree of biconnected components. In particular, each edge belongs to a unique biconnected component, although vertices can belong to several biconnected components.

	Assume that $\Gamma$ is connected. Given a nontrivial partition $V(\Gamma)=V\cup V^c$, the set $E(V,V^c)$ of edges joining a vertex in $V$ with one in $V^c$ is called a \emph{cut} of $\Gamma$. A \emph{bond} is a minimal cut, i.e., such that $\Gamma(V)$ and $\Gamma(V^c)$ are connected graphs. Given two cuts $E(V_1,V_1^c)$ and $E(V_2,V_2^c)$ with $V_1\cap V_2=\emptyset$  we define the sum $E(V_1,V_1^c)+E(V_2,V_2^c)$ as the cut $E(V_1\cup V_2,V_1^c\cap V_2^c)$. Note that the condition of $V_1$ and $V_2$ to be disjoint is necessary to have a well defined operation. In particular the equation $B_3=B_1+B_2$ will mean that $B_1=E(V_1,V_1^c)$ and $B_2=E(V_2,V_2^c)$, with $V_1$ and $V_2$ disjoint.\par
			A \emph{vertex weighted graph} is a pair $(\Gamma,w)$, where $\Gamma$ is a connected graph and $w$ is a function $w\colon V(\Gamma)\to\mathbb{Z}_{\geq0}$ (In this paper the labels will be nonnegative integers), called the \emph{labeling function}. The genus of the vertex weighted graph is defined as $\sum_{v\in V(\Gamma)} w(v)+b_1(\Gamma)$, where $b_1(\Gamma)$ is its first Betti number. A vertex weighted graph is \emph{stable} if every vertex of weight $0$ has valence at least $3$ and every vertex of weight $1$ has valence at least $1$.\par
	Given a subset $S\subset E(\Gamma)$, we define the graph $\Gamma/S$ as the graph obtained by contracting all edges in $S$. We say that a graph $\Gamma$ specializes to a graph $\Gamma'$ if there exists $S\subset E(\Gamma)$ such that $\Gamma'=\Gamma/S$.\par
	We finish this subsection with the following two simple lemmas.

\begin{Lem}
\label{lem:partition}
Let $\Gamma$ be a connected graph and $V:=\{v_1,\ldots,v_k\}$ be a subset of $V(\Gamma)$. Then there exists a (nonunique) partition $V_1,\ldots, V_k$ of $V(\Gamma)$ such that $v_i\in V_i$ and $\Gamma(V_i)$ is connected.
\end{Lem}
\begin{proof}
Let $\Gamma_1,\ldots, \Gamma_m$ be the connected components of $\Gamma(V^c)$. Since $\Gamma$ is connected, for each $j=1,\ldots,m$ there exists an edge $e_j\in E(\Gamma)$ with one vertex in $V(\Gamma_j)$ and the other one equal to $v_{l_j}$ for some $l_j=1,\ldots,k$. Then the result follows by defining
\[
V_i:=\{v_i\}\cup\bigcup_{j;l_j=i}V(\Gamma_j).
\]
\end{proof}
\begin{Lem}
\label{lem:bond}
	If $\Gamma$ is biconnected, then for any two edges $e_1$ and $e_2$ of $\Gamma$ there exists a bond $B$ of $\Gamma$ that contains both $e_1$ and $e_2$.
\end{Lem}
\begin{proof}
	Choose a cycle $v_1,v_2,\ldots, v_k$ of $\Gamma$ such that $e_1$ connects $v_1$ and $v_2$ and $e_2$ connects $v_j$ and $v_{j+1}$ for some $j=2,\ldots, k$, where $v_{k+1}=v_1$. Then by Lemma \ref{lem:partition} there exists a partition $V_1,\ldots, V_k$ of $V(\Gamma)$ such that $\Gamma(V_i)$ is connected and $v_i\in V_i$. Now, define $V:=\bigcup_{i=2}^j V_i$. Clearly $\Gamma(V)$ and $\Gamma(V^c)$ are connected hence $B=E(V,V^c)$ is a bond of $\Gamma$ that contains both $e_1$ and $e_2$.
\end{proof}

\subsection{Toric varieties}
   For the theory of toric varieties we refer to \cite{CLS}. We outline here some of the constructions we will use throughout the paper. 

A \emph{toric variety} $X$ of dimension $n$ is a normal variety with an embedding $T\hookrightarrow X$ of the $n$-dimensional torus $T$ together with an action of $T$ on $X$ that preserves the action of $T$ on itself. A map of toric varieties $X\to X'$ is called \emph{toric} if it respects the torus actions on both $X$ and $X'$.\par

Given a lattice $N=\mathbb{Z}^n$ and its dual $M:=N^\vee$, we define $N_{\mathbb{R}}:=N\otimes \mathbb{R}$ and $M_{\mathbb{R}}:=M\otimes \mathbb{R}$. Given a finite set $S\subset N_{\mathbb{R}}$ we define 
\[
\cone(S):=\left\{\sum_{s\in S}\lambda_ss|\lambda_s\geq0\right\}.
\]
A subset $\sigma\subset N_{\mathbb{R}}$ is called a \emph{polyhedral cone} if $\sigma=\cone(S)$ for some finite set $S\subset N_{\mathbb{R}}$. If there exists $S\subset N$ with $\sigma=\cone(S)$ then $\sigma$ is called \emph{rational}.\par
   Every polyhedral cone is the intersection of finitely many closed half-spaces. The \emph{dimension} of $\sigma$, denoted $\dim(\sigma)$, is the dimension of the minimal linear subspace containing $\sigma$, usually denoted $\text{span}(\sigma)$. The \emph{relative interior} $\text{int}(\sigma)$ is the interior of $\sigma$ inside this linear subspace. A \emph{face} of $\sigma$ is the intersection of $\sigma$ with some linear subspace $H\subset \mathbb{R}^n$ of codimension one such that $\sigma$ is contained in one of the closed half-spaces determined by $H$. A face of $\sigma$ is also a polyhedral cone. If $\tau$ is a face of $\sigma$ then we write $\tau\prec \sigma$. A face of dimension one of $\sigma$ is called a \emph{ray}. In the case that $\sigma$ is rational, such a ray has a minimal generator $u\in N$ called the \emph{ray generator}. We define $\sigma(1)$ as the set of ray generators of $\sigma$.	A rational cone is \emph{smooth} if its ray generators can be completed to a base of $N$.\par
	The dual cone of $\sigma$ is the cone
\[
\sigma^\vee:=\{m\in M_{\mathbb{R}}|\langle m,u\rangle\geq 0\text{ for every $u\in\sigma$}\}.
\]
Associated to a rational polyhedral cone $\sigma$ there is an affine toric variety $X(\sigma)$ whose ring of regular functions is $\mathbb{C}[\sigma^\vee\cap M]$. Hence the ring of rational functions is $\mathbb{C}[M]$. Moreover $X(\sigma)$ is smooth if and only if $\sigma$ is smooth.\par
A \emph{fan} $\Sigma$ is a collection of cones of $\mathbb{R}^n$ such that the following conditions hold
\begin{enumerate}[label=(\roman*)]
\item if $\sigma\in \Sigma$ then $\tau\in\Sigma$ for every $\tau\prec\sigma$;
\item if $\sigma,\tau\in\Sigma$ then $\sigma\cap\tau$ is a face of both $\sigma$ and $\tau$.
\end{enumerate}
The \emph{support} of $\Sigma$ is the union of all cones $\sigma\in\Sigma$. A fan is called \emph{complete} if its support is all $\mathbb{R}^n$. Associated to a fan $\Sigma$ there is a toric variety $X(\Sigma)$ obtained by glueing the affine toric varieties $X(\sigma)$ for every $\sigma\in\Sigma$. For every cone $\sigma\in\Sigma$ there exists an associated subvariety $V(\sigma)\subset X(\Sigma)$ of codimension $\dim(\sigma)$ which is invariant under the action of the torus. \par
\begin{Rem}
\label{rem:subvariety}  
	Given a cone $\tau\in\Sigma$ the invariant subvariety $V(\tau)$ is a toric variety, and its fan can be obtained as follows. Let $N_\tau:=\text{span}(\tau)\cap N$ be the sublattice induced by $\tau$ and $f\col N\to N(\tau):=N/N_\tau$ be the quotient. Then, by \cite[Proposition 3.2.7]{CLS}, the fan of $V(\tau)$ is the fan
\[
\text{Star}(\tau):=\{f(\sigma)|\tau\prec\sigma\in\Sigma\}.
\]
\end{Rem}
  Given two fans $\Sigma\subset N_{\mathbb{R}}$ and $\Sigma'\subset N'_{\mathbb{R}}$, and an integral map $r\col N_{\mathbb{R}}\to N'_{\mathbb{R}}$ such that for every $\sigma\in\Sigma$ there exists $\sigma'\in\Sigma'$ such that $r(\sigma)\subset\sigma'$, then there is an associated toric map $\phi_r\col X(\Sigma)\to X(\Sigma')$. If both $\Sigma$ and $\Sigma'$ are complete and $r$ is an inclusion (resp. surjection) then $\phi_r$ is an immersion (resp. surjection). \par

\begin{Def}
\label{def:star}
	Given a smooth cone $\sigma$ and a face $\tau\prec\sigma$ with dimension at least 2, we define the fan $\Sigma^\star_\sigma(\tau)$ as follows. Let $u_{\tau}$ be the sum of all ray generators of $\tau$, i.e.,
\[
u_\tau:=\sum_{u\in\tau(1)} u.
\]
Then we set
\[
\Sigma^\star_{\sigma}(\tau):=\{\cone(S)|S\subset\{u_\tau\}\cup\sigma(1), \tau(1)\not\subset S\}.
\]
Moreover, given a smooth fan $\Sigma$ and $\tau\in \Sigma$ we define the \emph{star subdivision} of $\Sigma$ relative to $\tau$ as the fan
\[
\Sigma^\star(\tau):=\{\sigma\in\Sigma|\tau\not\subset\sigma\}\cup\bigcup_{\tau\subset\sigma}\Sigma^\star_\sigma(\tau).
\]
\end{Def}
\begin{Rem}
\label{rem:unionstar}
Let $\Sigma_i$ be fans in $N_{\mathbb{R}}$, such that $\Sigma=\bigcup\Sigma_i$ is a fan, and let $\tau\in\Sigma$ be a cone with dimension at least two. Then
\[
\Sigma^\star(\tau)=\bigcup_{\tau\in\Sigma_i} \Sigma_i^\star(\tau)\cup\bigcup_{\tau\notin\Sigma_i} \Sigma_i.
\]
\end{Rem}
\begin{Rem}
\label{rem:prodstar}
Let $\Sigma_i$ be fans in $N_{i,\mathbb{R}}$ and $\Sigma:=\prod \Sigma_i$ be the induced fan in $N_{\mathbb{R}}$, where $N=\prod N_i$. Consider $\tau\in \Sigma_1$ and let $\tau'=\tau\times\prod_{i>1}\{0\}$ be the image of $\tau$ via the natural inclusion $N_{1,\mathbb{R}}\to N_{\mathbb{R}}$. Then 
\[
\Sigma^\star(\tau')=\Sigma_1^\star(\tau)\times\prod_{i>1}\Sigma_i.
\]
\end{Rem}
	The blowup of a toric variety $X(\Sigma)$ along the invariant subvariety $V(\tau)$ for $\tau\in\Sigma$ is also a toric variety with fan $\Sigma^\star(\tau)$, see \cite[Section 3.3]{CLS}. \par
	Given a cone $\sigma\subset N_\mathbb{R}$ and a vector $v\in N_\mathbb{R}\backslash \text{span}(\sigma)$, we define 
\[
\cone(\sigma,v)=\cone(\{v\}\cup\sigma(1)).
\]
 Given a proper primitive inclusion $N'\subset N$, a fan $\Sigma$ in $N'_\mathbb{R}$ and a vector $v\in N\backslash N'$ we define the fan
\begin{equation}
\label{eq:fanv}
\text{fan}(\Sigma,v)=\Sigma\cup\bigcup_{\sigma\in\Sigma}\{\cone(\sigma,v)\}.
\end{equation}
\begin{Prop}
\label{prop:fanv}
Given a proper primitive inclusion $N'\subset N$, a fan $\Sigma$ in $N'_\mathbb{R}$, a vector $v\in N\backslash N'$ and $\tau\in\Sigma$ of dimension at least two, we have
\[
\textnormal{fan}(\Sigma,v)^\star(\tau)=\textnormal{fan}(\Sigma^\star(\tau),v).
\]
\end{Prop}
\begin{proof}
By Definition \ref{def:star}, we have that
\begin{align*}
\text{fan}(\Sigma,v)^\star(\tau)=&\{\sigma\in\text{fan}(\Sigma,v)|\tau\not\subset\sigma\}\cup\underset{\sigma\in\text{fan}(\Sigma,v)}{\underset{\tau\subset\sigma}{\bigcup}}\Sigma^\star_\sigma(\tau)\\
                               =&\{\sigma\in\Sigma|\tau\not\subset\sigma\}\cup \{\text{cone}(\sigma,v)|\sigma\in\Sigma,\tau\not\subset\sigma\}\cup\\
															&\cup\underset{\sigma\in\Sigma}{\underset{\tau\subset\sigma}{\bigcup}}\Sigma^\star_\sigma(\tau)\cup\underset{\sigma\in\Sigma}{\underset{\tau\subset\sigma}{\bigcup}}\Sigma^\star_{\text{cone}(\sigma,v)}(\tau).\\
\end{align*}
Since 
\begin{align*}
\Sigma^\star_{\text{cone}(\sigma,v)}(\tau)=&\{\cone(S)|S\subset\{u_\tau\}\cup\sigma(1)\cup\{v\}, \tau(1)\not\subset S\}\\
=&\{\cone(S)|S\subset\{u_\tau\}\cup\sigma(1), \tau(1)\not\subset S\}\cup\\
&\cup\{\cone(S)|S\subset\{u_\tau\}\cup\sigma(1)\cup\{v\}, \tau(1)\not\subset S, v\in S\}\\
=&\Sigma^\star_{\sigma}(\tau)\cup\{\cone(S\cup\{v\})|S\subset\{u_\tau\}\cup\sigma(1), \tau(1)\not\subset S\}\\
=&\Sigma^\star_{\sigma}(\tau)\cup\{\text{cone}(\sigma',v)| \sigma'\in\Sigma^\star_\sigma(\tau)\},
\end{align*}
we have
\begin{align*}
\text{fan}(\Sigma,v)^\star(\tau)=&\{\sigma\in\Sigma|\tau\not\subset\sigma\}\cup \underset{\sigma\in\Sigma}{\underset{\tau\subset\sigma}{\bigcup}}\Sigma^\star_\sigma(\tau)\cup\\
&\cup\{\text{cone}(\sigma,v)|\sigma\in\Sigma, \tau\not\subset\sigma\}\cup\underset{\sigma\in\Sigma}{\underset{\tau\subset\sigma}{\bigcup}}\{\text{cone}(\sigma',v)| \sigma'\in\Sigma^\star_\sigma(\tau)\}\\
=&\Sigma^\star(\tau)\cup\{\text{cone}(\sigma',v)|\sigma'\in\Sigma^\star(\tau)\}\\
=&\text{fan}(\Sigma^\star(\tau),v),
\end{align*}
which finishes the proof.
\end{proof}
\subsection{Enriched structures}
\label{sec:maino}
  In this section we give an overview of the theory of enriched curves introduced by Main\`o in \cite{maino}. We also refer to \cite{EM} for a different approach to the theory.\\

	 A \emph{curve} is a connected, projective, reduced scheme of dimension 1 over $\mathbb{C}$ all of whose singularities are nodes. A curve is \emph{stable} if its automorphism group is finite.
	Let $Y$ be a stable curve with irreducible components $Y_1,\ldots, Y_\gamma$. For every $i=1,\ldots,\gamma$ define $\Delta_{Y_i}:=Y_i\cap Y_i^c$, where $Y_i^c:=\overline{Y\backslash Y_i}$. A \emph{regular smoothing} of $Y$ is a proper flat family $\Y\to B$, where $\Y$ is smooth, $B=\Spec(\mathbb{C}[[t]])$ and $Y$ is isomorphic to the special fiber.
\begin{Def}
An \emph{enriched structure} over $Y$ is a $\gamma$-tuple $\underline{L}:=(L_1,L_2,\ldots, L_\gamma)$ of invertible sheaves on $Y$ such that there exists a regular smoothing $\Y\to B$ of $Y$ with $L_i=\O_{\Y}(Y_i)|_Y$ for $i=1,\ldots, \gamma$. An \emph{enriched curve} is a pair $(Y,\underline{L})$ where $Y$ is a stable curve and $\underline{L}$ is an enriched structure on $Y$.
\end{Def}
  Equivalently, by \cite[Proposition 3.16]{maino}, one can define an enriched structure as a $\gamma$-tuple $\underline{L}:=(L_1,\ldots,L_\gamma)$ of invertible sheaves on $Y$ such that 
\begin{enumerate}[label=(\roman*)]
\item $L_i|_{Y_i}=\O_{Y_i}(-\Delta_{Y_i})$ and $L_i|_{Y_i^c}=\O_{ Y_i^c}(\Delta_{Y_i})$.
\item $\otimes_{i=1}^\gamma L_i=\O_Y.$
\end{enumerate}
For each $i=1,\ldots,\gamma$ we have a natural exact sequence
\[
0\lra (\mathbb{C}^*)^{|\Delta_{Y_i}|-c_i}\lra \text{Pic}(Y)\lra \text{Pic}(Y_i)\times \text{Pic}(Y_i^c)\lra0,
\]
where $c_i$ is the number of connected components of the subcurve $Y_i^c$. Since, by Condition (i) above, the invertible sheaf $L_i$ has fixed restrictions to $Y_i$ and $Y_i^c$ the torus $(\mathbb{C}^*)^{|\Delta_{Y_i}|-1}$ parametrizes all the possible choices for $L_i$. Moreover, imposing condition (ii), we get equations describing the variety of enriched structures in $\prod (\mathbb{C}^*)^{|\Delta_{Y_i}|-1}$. Indeed, the space of enriched structures can be described by the following proposition, proved in \cite[Proposition 3.14]{maino}
\begin{Prop}
Let $Y$ be a stable curve with irreducible components $Y_1,\ldots, Y_\gamma$ and $\delta$ nodes, $\rho$ of which are external nodes, i.e., belong to some $\Delta_{Y_i}$. Then all possible enriched structures on $Y$ form a principal $(\mathbb{C}^\star)^n$-homogeneous space of dimension $n=\rho+\gamma-\sum_{i=1}^\gamma c_i-1$, where $c_i$ is the number of connected components of the subcurve $Y_i^c$.
\end{Prop}

  Let $\V_Y$ be the space of versal deformations of $Y$ (for more details see \cite[Section 11.3]{ACG}). A stable curve is called \emph{relevant} if its dual graph is biconnected. Consider the subspaces of $\V_Y$ called the \emph{relevant loci}
\[
\R_i=\{[C]\in \V_Y| C\text{ is a relevant curve, with $i$ nodes}\}
\]
for $i=1,\ldots, 3g-3$.
For some $i$ the relevant locus $\R_i$ can be empty. Therefore let $i_m>i_{m-1}>\ldots>i_1$ be the indices such that $\R_{i_j}$ is nonempty.  Also note that $\overline{\R_1}$ is a Cartier divisor in $\V_Y$ usually denoted $\Delta_0$. We then define the \emph{Main\`o blowup}  $\pi\col B\to\V_Y$ as the chain of blowups:
\[
\pi\col B:=B_1\xrightarrow{\phi_1} B_2\xrightarrow{\phi_2}\cdots \xrightarrow{\phi_{m-2}} B_{m-1}\xrightarrow{\phi_{m-1}} B_m \xrightarrow{\phi_m} B_{m+1}:=\V_Y
\]
where $\phi_j\col B_j\to B_{j+1}$ is the blowup of 
\[
\overline{\pi_{j+1}^{-1}(\R_{i_j})}
\]
where $\pi_{n+1}=\text{id}_{\V_Y}$ and $\pi_{j-1}=\phi_{j-1}\pi_j$. Define $\widetilde{R}_{i_j}:=\overline{\pi^{-1}(R_{i_j})}$ and
\[
B':=\begin{cases}
 B\backslash \bigcup_{j=1}^{m-1} \widetilde{R}_{i_j}&  \text{if $Y$ has biconnected dual graph;}\\
  B\backslash \bigcup_{j=1}^{m} \widetilde{R}_{i_j}& \text{otherwise.}\\
 \end{cases}
\]
Also let $\pi':=\pi|_{B'}$.\par
    The following theorem gives another description of the space of enriched structures using the deformation space of $Y$. For a proof see \cite[Theorem 5.4]{maino}.
\begin{Thm}
Let $Y$ be a stable curve. Then the space of enriched structures on $Y$ is isomorphic to $\pi'^{-1}([Y])$.
\end{Thm}

The construction above can be reproduced to blowup $\overline{\M}_g$ instead of $\V_Y$, given rise to a space $\E_g$ and to a map $\E_g\to\overline{\M}_g$. The space $\E_g$ parametrizes genus-$g$ stable curves with enriched structures. Moreover, if $Y$ is a stable curve with no nontrivial automorphisms, then the fiber of $\E_g\to\overline{\M}_g$ over $[Y]\in \overline{\M}_g$ is the same as the fiber $\pi'^{-1}([Y])$. We refer to \cite[Chapter 5]{maino} for more details. \par

\subsection{Stacky fans and tropical curves}
\label{sub:stacky}
  In this section we review the definition of stacky fans and the cell description of the moduli space $M_g^{trop}$ parametrizing tropical curves. We refer the reader to \cite{BMV} for more details.

\begin{Def}
\label{def:stackyfan}
 Let $\{K_j\subset \mathbb{R}^{m_j}\}_{j\in J}$ be a finite collection of rational open polyhedral cones such that $\dim K_j=m_j$. Moreover, for each cone $K_j\subset \mathbb{R}^{m_j}$, let $G_j$ be a group and $\rho_j\col G_j\to \text{GL}_{m_j}(\mathbb{Z}^{m_j})$ a homomorphism such that $\rho_j(G_j)$ stabilizes the cone $K_j$ under its natural action on $\mathbb{R}^{m_j}$. Therefore $G_j$ acts on $K_j$ (resp. $\overline{K}_j$), via the homomorphism $\rho_j$, and we denote the quotient by $K_j/G_j$ (resp. $\overline{K}_j/G_j$), endowed with the quotient topology. A topological space $\K$ is said to be a \emph{stacky fan} with cells $\{K_j/G_j\}_{j\in J}$ if there exist continuous maps $\alpha_j\col \overline{K}_j/G_j\to \K$ satisfying the following properties:

\begin{enumerate}[label=(\roman*)]
\item The restriction of $\alpha_j$ to $K_j/G_j$ is an homeomorphism onto its image;
\item $\K=\coprod_j \alpha_j(K_j/G_j)$ (set-theoretically);
\item For any $j_1, j_2\in J$, the natural inclusion map 
\[
\alpha_{j_1}(\overline{K}_{j_1}/G_{j_1})\cap\alpha_{j_2}(\overline{K}_{j_2}/G_{j_2})\hookrightarrow \alpha_{j_2}(\overline{K}_{j_2}/G_{j_2})
\]
is induced by an integral linear map $L\col\mathbb{R}^{m_{j_1}}\to\mathbb{R}^{m_{j_2}}$, i.e., there exists a commutative diagram
\[
\SelectTips{cm}{11}
\begin{xy}<14pt,0pt>:
\xymatrix@C-=0.5cm{\alpha_{j_1}(\overline{K}_{j_1}/G_{j_1})\cap\alpha_{j_2}(\overline{K}_{j_2}/G_{j_2})\UseTips\ar@{^{(}->}[r] \ar@{^{(}->}[dr] &\alpha_{j_1}(\overline{K}_{j_1}/G_{j_1}) & \ar@{->>}[l] \overline{K}_{j_1} \ar@{^{(}->}[r]\ar[d] & \mathbb{R}^{m_{j_1}}\ar[d]^{L} \\
 &\alpha_{j_2}(\overline{K}_{j_2}/G_{j_2}) & \ar@{->>}[l] \overline{K}_{j_2} \ar@{^{(}->}[r] & \mathbb{R}^{m_{j_2}}
}
\end{xy}
\]
\end{enumerate}
\end{Def}
The \emph{dimension} of $\K$ is the greatest dimension of its cells. We say that a cell is \emph{maximal} if it is not contained in the closure of any other cell. The stacky fan $\K$ is said to be of \emph{pure dimension} if all its maximal cells have dimension equal to the dimension of $\K$. A \emph{generic point} of $\K$ is a point contained in a cell of maximal dimension.\par
  Assume that $\K$ is a stacky fan of pure dimension $n$. The cells of dimension $n-1$ are called \emph{codimension one cells}. The stacky fan $\K$ is said to be \emph{connected through codimension one} if for any two maximal cells $K_{j}/G_{j}$ and $K_{j'}/G_{j'}$ one can find a sequence of maximal cells 
\[
K_{j}/G_{j}=:K_{j_0}/G_{j_0},K_{j_1}/G_{j_1},\ldots,K_{j_r}/G_{j_r}:=K_{j'}/G_{j'}
\]
such that for any $i:=0,\ldots, r-1$ the two consecutive maximal cells $K_{j_i}/G_{j_i}$ and $K_{j_{i+1}}/G_{j_{i+1}}$ have a common codimension one cell in their closure.

\begin{Def}
Let $\K$ and $\K'$ be two stacky fans with cells $\{K_j/G_j\}$ and $\{K'_i/H_i\}$ where $\{K_j\subset \mathbb{R}^{m_j}\}$ and $\{K'_i\subset \mathbb{R}^{m'_i}\}$, respectively. A continuous map $\pi:\K\to \K'$ is said to be a \emph{map of stacky fans} if for every cell $K_j/G_j$ of $\K$ there exists a cell $K'_i/H_i$ of $\K'$ such that
\begin{enumerate}[label=(\roman*)]
\item $\pi(K_j/G_j)\subset K'_i/H_i$;
\item $\pi\col K_j/G_j\to K'_i/H_i$ is induced by an integral linear function $L_{j,i}\col\mathbb{R}^{m_j}\to\mathbb{R}^{m'_i}$, i.e, there exists a commutative diagram
\[
\SelectTips{cm}{11}
\begin{xy}<16pt,0pt>:
\xymatrix@+=1cm{K_j/G_j\ar^{\pi}[d] &\ar@{->>}[l] K_j\ar@{^{(}->}[r]\ar^{L_{j,i}}[d]  & \mathbb{R}^{m_j}\ar^{L_{j,i}}[d]\\
K'_i/H_i&\ar@{->>}[l] K'_i\ar@{^{(}->}[r]  & \mathbb{R}^{m'_i}}
\end{xy}
\]
\end{enumerate}
 We say that the map $\pi$ is \emph{full} if it sends every cell $K_j/G_j$ of $\K$ surjectively onto some cell $K'_i/H_i$ of $\K'$. We say that $\pi$ is of \emph{degree one} if for every generic point $Q\in K'_i/H_i\subset \K'$ the inverse image $\pi^{-1}(Q)$ consists of a single point $P\in K_j/G_j\subset \K$ and the integral linear function $L_{j,i}$ inducing $\pi\col K_j/G_j\to K'_i/H_i$ is primitive (i.e., $L_{j,i}^{-1}(\mathbb{Z}^{m_i'})\subset\mathbb{Z}^{m_j}$).
\end{Def}

   A stable tropical curve is a 3-tuple $(\Gamma,w,l)$, where $(\Gamma,w)$ is a vertex weighted stable graph and $l$ is a function $l\col E(\Gamma)\to \mathbb{R}_{>0}$ called the \emph{length function}. The genus of a tropical stable curve is the genus of the underlying vertex weighted graph.\par
	 
Given a vertex weighted graph $(\Gamma,w)$, to parametrize stable tropical curves with $(\Gamma,w)$ as its underlying vertex weighted graph, we just have to assign lengths to each edge of $\Gamma$. Then it is clear that such a datum is parametrized by the space
\[
\mathbb{R}_{>0}^{E(\Gamma)}/\Aut(\Gamma,w).
\]
In fact, we have the following theorem, proved in \cite[Theorem 3.2.4]{BMV}.

\begin{Thm}
There exists a stacky fan $M_g^{trop}$ whose points are in bijection with stable tropical curves of genus $g$, and whose cells are of the form
\[
\mathbb{R}_{>0}^{E(\Gamma)}/\Aut(\Gamma,w).
\]
\end{Thm}

\section{Enriched graphs}
\label{sec:enrichedgraphs}

 In this section we define enriched structures on a graph and prove several results that will be used in the next sections. The definition of an enriched structure is recursive. The propositions in the first part of the section will be devoted to give a more concrete description of this notion. In the second part of the section we prove Propositions \ref{prop:unionK} and \ref{prop:star} that will translate to Theorems \ref{thm:maintrop} and \ref{thm:mainEX}, respectively.

\subsection{Definition and basic properties.} 
   
\begin{Def}
\label{def:enrichedgraph}
An \emph{enriched graph} is a pair $(\Gamma,p)$, where $\Gamma$ is a graph and $p$ is a preorder on the set of edges $E(\Gamma)$, satisfying one of the following conditions:
\begin{enumerate}[label=(\arabic*)]
\item The graph $\Gamma$ has only one edge (or is without edges) and the preorder is the trivial one;
\item The graph $\Gamma$ is biconnected, has at least two edges and there is a nonempty lower set $S\subset E(\Gamma)$ such that $e\<_p e'$ for all $e\in S$ and $e'\in E(\Gamma)$, and such that $(\Gamma/S,p|_{E(\Gamma/S)})$ is an enriched graph;
\item The graph $\Gamma$ has biconnected components $\Gamma_1, \ldots, \Gamma_m$, with $m\geq 2$, so that $E(\Gamma)=\coprod E(\Gamma_i)$, the pairs $(\Gamma_i,p|_{E(\Gamma_i)})$ are enriched graphs, and edges $e, e'$ belonging to different components $\Gamma_i$ and $\Gamma_j$ are incomparable by $p$.
\end{enumerate}
Recall that a preorder $p$ on $E(\Gamma)$ induces a equivalence relation $\sim_p$, as defined in Section \ref{sec:preorder}. Then the second condition is equivalent to the following condition:
\begin{enumerate}
\item[$(2')$] The graph $\Gamma$ is biconnected, has at least two edges and there is an equivalence class $[e]_p$, such that $e\<_p e'$ for every $e'\in E(\Gamma)$ and $(\Gamma/[e]_p,p|_{E(\Gamma/[e]_p)})$ is an enriched graph.\par
\end{enumerate}
\end{Def}

If $(\Gamma,p)$ is an enriched graph, we say that $p$ is an \emph{enriched structure} on $\Gamma$. Moreover, if $p$ is a partial order, then we say that the enriched graph $(\Gamma,p)$ is \emph{generic} and that $p$ is a \emph{generic} enriched structure on $\Gamma$. We also note that the second condition implies that, if $\Gamma$ is biconnected, then $E(\Gamma)$ is irreducible with respect to the topology induced by $p$, while the third condition implies that $E(\Gamma_i)$ are the connected components of $E(\Gamma)$ in the topology induced by $p$. Furthermore, if $\Gamma$ is biconnected there is a canonical enriched structure $p$ on $\Gamma$ given by $e\sim_p e'$ for every $e,e'\in \Gamma$ (in fact $(\Gamma,p)$ clearly satisfies Condition (2) in Definition \ref{def:enrichedgraph} for $S=E(\Gamma)$).\par

\begin{Exa}
Let $\Gamma$ be the graph in Figure \ref{fig:graph}. Let us construct enriched structures $p_1$, $p_2$ and $p_3$ on $\Gamma$. Since $\Gamma$ is biconnected, we first choose a set $S\subset E(\Gamma)$ to be a lower set and then iterate the process following Conditions (1), (2) and (3) in Definition \ref{def:enrichedgraph}.\par

\begin{figure}[htb]
\[
\begin{xy} <30pt,0pt>:
(0,0)*{\scriptstyle\bullet}="a"; 
(-1,-1.5)*{\scriptstyle\bullet}="b";
(1,-1.5)*{\scriptstyle\bullet}="d";
"a"+0;"b"+0**\crv{(-0.8,-0.8)};
"a"+0;"d"+0**\crv{(0.8,-0.8)};
"b"+0;"d"+0**\crv{(0,-1)}; 
"b"+0;"d"+0**\crv{(0,-2)};
"b"+(1,0,4)*{e_1};
"b"+(1,-0,4)*{e_2};
"a"+(-0.8,-0,6)*{e_3};
"a"+(0.8,-0,6)*{e_4};
\end{xy}
\]
\caption{The graph $\Gamma$.}
\label{fig:graph}
\end{figure}
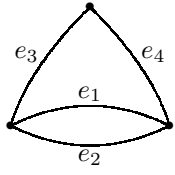

\begin{itemize}
   \item Choose $S=\{e_1\}$. The graph $\Gamma_1:=\Gamma/\{e_1\}$ is not biconnected. Then, by Condition (3), it is enough to find enriched structures for each one of its biconnected components $\Gamma_2$ and $\Gamma_3$, where $E(\Gamma_2)=\{e_2\}$ and $E(\Gamma_3)=\{e_3,e_4\}$. The enriched structure on $\Gamma_2$ is the trivial one by Condition (1). Now for $\Gamma_3$ we have to choose again a lower set of $E(\Gamma_3)=\{e_3,e_4\}$. We choose $\{e_3\}$ to be such a lower set. The enriched structure on $\Gamma_3/\{e_3\}$ is the trivial one. Hence, an enriched structure $p_1$ on $\Gamma$ is given by the relations:
\[
e_1\<_{p_1}e_2 \quad \quad\text{and}\quad\quad e_1\<_{p_1}e_3\<_{p_1}e_4.
\]
  \item Choose $S=\{e_3\}$. The graph $\Gamma_1:=\Gamma/\{e_3\}$ is still biconnected. We have to choose a lower set of $E(\Gamma_1)$. We choose $\{e_1\}$ to be such a lower set. Now the graph $\Gamma_2:=\Gamma_1/\{e_1\}$ is not biconnected. The biconnected components of $\Gamma_2$ have only one edge, hence the enriched structures on them are trivial. Therefore, an enriched structure $p_2$ on $\Gamma$ is given by the relations:
\[
e_3\<_{p_2}e_1\<_{p_2}e_2\quad\quad \text{and} \quad \quad e_3\<_{p_2}e_1\<_{p_2}e_4.
\]
  \item Choose $S=\{e_1,e_3\}$. The graph $\Gamma_1:=\Gamma/S$ is not biconnected. The biconnected components of $\Gamma_1$ have only one edge, hence the enriched structures on them are trivial. Therefore, an enriched structure $p_3$ on $\Gamma$ is given by the relations:
\[
e_1\sim_{p_3}e_3\<_{p_3}e_2\quad\quad \text{and}\quad\quad  e_1\sim_{p_3}e_3\<_{p_3}e_4.
\]
\end{itemize}
We note that $p_1$ and $p_2$ are generic, while $p_3$ is not.
\end{Exa}

\begin{Exa}
\label{ex:circular}
If $\Gamma$ is a circular graph, then a generic enriched structure is simply a total order of $E(\Gamma)$. 
\end{Exa}
\begin{Exa}
\label{ex:2c}
If $\Gamma$ is a graph with two vertices and no loops, then an enriched structure $p$ is simply the choice of a subset $S\subset E(\Gamma)$, where $e\<_p e'$ for every $e\in S$ and $e'\in E(\Gamma)$.
\end{Exa}
  We will often use the following key proposition without explicitly mentioning it.
\begin{Prop}
Given an enriched graph $(\Gamma,p)$ and a nonempty lower set  $S\subset E(\Gamma)$ with respect to $p$, then $(\Gamma/S,p|_{E(\Gamma/S)})$ is an enriched graph.
\end{Prop}
\begin{proof}
The proof is by induction on the number of edges of $\Gamma$. If $\Gamma$ has only one edge, or no edges, the result is trivial.\par
If $\Gamma$ is not biconnected, then it has biconnected components $\Gamma_1,\ldots, \Gamma_m$, and $\Gamma/S$ has biconnected components $\Gamma_1',\ldots, \Gamma_{k}'$ where each $\Gamma_j'$ is a biconnected component of some unique $\Gamma_i/(E(\Gamma_i)\cap S)$. By Condition $(3)$ in Definition \ref{def:enrichedgraph} we have that $(\Gamma_i,p|_{E(\Gamma_i)})$ is an enriched graph and the set $E(\Gamma_i)\cap S$ is a lower set. Hence, by the induction hypothesis, $(\Gamma_i/(E(\Gamma_i)\cap S),p|_{E(\Gamma_i)\setminus S})$ is an enriched graph, and again by Condition $(3)$ we have that $(\Gamma_i',p|_{E(\Gamma_i')})$ is an enriched graph. Finally, Condition (3) one more time implies that $(\Gamma/S,p|_{E(\Gamma/S)})$ is an enriched graph.\par

If $\Gamma$ is biconnected, then, by Condition $(2')$ in Definition \ref{def:enrichedgraph}, there exists an equivalence class $[e]_p$ such that $e\<_p e'$ for every $e'\in E(\Gamma)$; in particular $[e]_p\subset S$. Contracting all edges in $[e]_p$, we have an enriched graph $(\Gamma/[e]_p,p|_{E(\Gamma/[e]_p)})$ and a lower set $S\backslash [e]_p\subset E(\Gamma/[e]_p)$. If $S\backslash [e]_p$ is empty, then $S=[e]_p$, and the result follows by Condition ($2'$) in Definition \ref{def:enrichedgraph}. Otherwise, $\Gamma/S=(\Gamma/[e]_p)/(S\backslash [e]_p)$, and then the result follows by the induction hypothesis applied to $(\Gamma/[e]_p,p|_{E(\Gamma/[e]_p)})$ and $S\backslash [e]_p$.
\end{proof}

\begin{Prop}
\label{prop:tree}
Given an enriched graph $(\Gamma,p)$, if $e_1$, $e_2$ and $e_3$ are edges in $\Gamma$ such that $e_1\<_p e_3$ and $e_2\<_p e_3$ then either $e_1\<_p e_2$ or $e_2\<_p e_1$. 
\end{Prop}
\begin{proof}
The proof is by induction on the number of edges. If $\Gamma$ has only one edge, or no edges, the result is trivial.\par
  If $\Gamma$ is not biconnected, then, by Condition $(3)$ in Definition \ref{def:enrichedgraph}, the edges $e_1$, $e_2$ and $e_3$ all belong to the same biconnected component, then the induction hypothesis applies.\par
	If $\Gamma$ is biconnected, then, by Condition $(2)$ in Definition \ref{def:enrichedgraph}, there exists a lower set $S\subset E(\Gamma)$ such that $e\<_p e'$ for every $e\in S$ and $e'\in E(\Gamma)$. If $e_i\in S$ for some $i=1,2,3$, then the result is clear. Otherwise the edges $e_1$, $e_2$ and $e_3$ belong to the enriched graph $(\Gamma/S,p|_{E(\Gamma/S)})$ and the result follows again by the induction hypothesis.
\begin{Cor}
\label{cor:hasse}
Given an enriched graph $(\Gamma,p)$, the connected components of the Hasse diagram of $E(\Gamma)_p$ are rooted trees, where the roots are the classes of the elements of $E(\Gamma)$ which are minimal with respect to $p$. 
\end{Cor}\end{proof}

Given an enriched graph $(\Gamma,p)$, we define the \emph{rank} of $p$ as 
\[
\rank(p):=\#E(\Gamma)_p. 
\]
If $(\Gamma,p)$ is generic, then $\rank(p)=|E(\Gamma)|$.
We say that two equivalence classes $[e_1]_p$ and $[e_2]_p$ are \emph{consecutive} if there is a directed edge from $[e_1]_p$ to  $[e_2]_p$ in the Hasse diagram of $E(\Gamma)_p$, i.e., $e_1<_p e_2$ and there exists no edge $e_3$ such that $e_1<_p e_3<_p e_2$.\par
  The following proposition describes how enriched structures behave when restricted to bonds of the graph $\Gamma$. This proposition will later translate to Proposition \ref{prop:map}.
\begin{Prop}
\label{prop:bond}
Given an enriched graph $(\Gamma,p)$ and a bond $B$ of $\Gamma$, there exists a nonempty lower set $T\subset B$ with respect to $p|_B$ such that $e\<_p e'$ for every $e\in T$ and $e'\in B$.
\end{Prop}
\begin{proof}
The proof is by induction on the number of edges. If $\Gamma$ has only one edge, or no edges, the result is trivial.\par
  If $\Gamma$ is not biconnected then a bond of $\Gamma$ is a bond of some biconnected component, then the result follows by induction.\par
	If $\Gamma$ is biconnected, then, by Condition $(2)$ in Definition \ref{def:enrichedgraph}, there exists a lower set $S\subset E(\Gamma)$ such that $e\<_p e'$ for every $e\in S$ and $e'\in E(\Gamma)$. If $S\cap B$ is nonempty then the set $T:=S\cap B$ is clearly a lower set with respect to $p|_B$ that satisfies $e\<_p e'$ for every $e\in T$ and $e'\in B$. If $S\cap B=\emptyset$ then $B$ is a bond of $\Gamma/S$ and the result follows by the induction hypothesis.
\end{proof}

  Conversely, let $\mathcal{B}$ be the set of all bonds of $\Gamma$. Choose a collection $\T=(T_B)_{B\in\B}$ of nonempty subsets $T_B\subset B$ for every $B\in\B$, and define the preorder $p_\T$ as the transitive closure of the relations $e\<_{p_\T} e'$ for every $B\in\B$, $e\in T_B$, $e'\in B$.

\begin{Prop}
\label{prop:bond1}
Let $\Gamma$ be a graph and let $\B$ be the set of all bonds of $\Gamma$. Given a collection $\T=(T_B)_{B\in\B}$ of nonempty subsets $T_B\subset B$, the induced preorder $p_\T$ is an enriched structure on $\Gamma$.
\end{Prop}
\begin{proof}
The proof is by induction on the number of edges. If $\Gamma$ has only one edge, or no edges, the result is trivial.\par
If $\Gamma$ is not biconnected, then, by Condition $(3)$ in Definition \ref{def:enrichedgraph} and the fact that each bond is contained in a biconnected component, the induction hypothesis applies.\par
If $\Gamma$ is biconnected, then let $S\subset E(\Gamma)$ be a minimal nonempty lower set with respect to $p_T$, i.e., a lower set $S$ such that if $e,e'\in S$ then $e\sim_{p_\T} e'$. We will now check condition (2) in Definition \ref{def:enrichedgraph}. Fix $e\in S$ and $e'\in E(\Gamma)$. By Lemma \ref{lem:bond} there exists a bond $B$ such that $e,e'\in B$. If $e\in T_B$, then $e\<_{p_\T}e'$ and we are done. If $e\notin T_B$, then there exists $e_1\in T_B$ and hence $e_1\<_{p_\T} e$ and $e_1\<_{p_\T}e'$. This implies that $e_1\in S$, because $S$ is a lower set, whence $e_1\sim_{p_\T} e$, because $S$ is minimal, and we get that $e\<_{p_\T} e'$.\par
  Finally, contract all edges in $S$. We have that every bond of $\Gamma/S$ is a bond of $\Gamma$ which does not intersect $S$. Hence the collection $\T$ restricts to a collection of nonempty subsets of bonds of $\Gamma/S$ which, by the induction hypothesis, induces an enriched structure $q$ on $\Gamma/S$. Clearly $q=p_\T|_{E(\Gamma/S)}$ which proves that $p_{\T}$ is an enriched structure on $\Gamma$.
\end{proof}

  Now, we will define the meaning of specialization for enriched graphs in analogy to the specialization of graphs.
\begin{Def}
\label{def:specializes}
Given enriched graphs $(\Gamma,p)$ and $(\Gamma',p')$ we say that $(\Gamma,p)$ \emph{specializes to} $(\Gamma',p')$ if the following conditions hold.
\begin{enumerate}
\item The graph $\Gamma$ specializes to $\Gamma'$. In particular $E(\Gamma')\subset E(\Gamma)$.
\item For $e, e'\in E(\Gamma)$ such that $e\<_p e'$ and $e'\notin E(\Gamma')$, we have $e\notin E(\Gamma')$.
\item For $e,e'\in E(\Gamma')$ such that $e\<_p e'$, we have $e\<_{p'} e'$.
\end{enumerate}
In particular there is a (possibly empty) lower set $S\subset E(\Gamma)$ with respect to $p$ such that $\Gamma'=\Gamma/S$.  We use the notation $(\Gamma,p)\leadsto (\Gamma',p')$ to mean that $(\Gamma,p)$ specializes to $(\Gamma',p')$.
\end{Def}
We note that given a lower set $S$ of an enriched graph $(\Gamma,p)$, then $(\Gamma,p)$ specializes to the enriched graph $(\Gamma/S, p|_{E(\Gamma/S)})$. Moreover if $(\Gamma,p)$ specializes to $(\Gamma',p')$ and $(\Gamma',p')$ specializes to $(\Gamma'',p'')$ then $(\Gamma,p)$ specializes to $(\Gamma'',p'')$.\par
Given a specialization $(\Gamma,p)\leadsto(\Gamma,p')$ there is a surjective map $E(\Gamma)_p\to E(\Gamma)_{p'}$; in particular $\rank(p')\leq\rank(p)$. The case where $\rank(p')=\rank(p)-1$ will be an important one because every specialization can be obtained as a chain of specializations each one of which decreases the rank by one, as Proposition \ref{prop:simple} will illustrate.

\begin{Lem-Def}
\label{lemdef:simple}
Given an enriched graph $(\Gamma,p)$ and two consecutive equivalence classes $[e_1]_p$ and $[e_2]_p$, there exists an enriched graph $(\Gamma,p')$ such that $e\<_{p'} e'$ if and only if $e\<_p e'$ or $\{e,e'\}\subset[e_1]_p\cup[e_2]_p$. In particular $(\Gamma,p)$ specializes to $(\Gamma,p')$ and we call such a  specialization simple. 
\end{Lem-Def}
\begin{proof}
Define $p'$ as in the statement. We just have to prove that $p'$ is an enriched structure on $\Gamma$. We proceed by induction on the number of edges of $\Gamma$. The minimum number of edges of $\Gamma$ is two, in which case the result is trivial.\par
If $\Gamma$ is not biconnected, then $[e_1]_p$ and $[e_2]_p$ are sets of edges of the same biconnected component $\Gamma_1$ of $\Gamma$, and the result follows by the induction hypothesis applied to $\Gamma_1$.\par
If $\Gamma$ is biconnected, by Condition $(2')$ in Definition \ref{def:enrichedgraph} there exists an edge $e$ such that $e\<_p e'$ for every $e'\in E(\Gamma)$ and $(\Gamma/[e]_p,p|_{E(\Gamma/[e]_p)})$ is an enriched graph. We now divide the proof in two cases.\par
In the first case we have $[e_1]_p=[e]_p$. Then, by the definition of $p'$, we have $[e]_{p'}=[e_1]_p\cup[e_2]_p$, and $e$ satisfies $e\<_{p'} e'$ for every $e'\in E(\Gamma)$. Since $[e_1]_p$ and $[e_2]_p$ are consecutive, we see that $[e]_{p'}=[e_1]_p\cup [e_2]_p$ is a lower set with respect to $p$, hence $(\Gamma/[e]_{p'},p|_{E(\Gamma/[e]_{p'})})$ is an enriched graph. Moreover we have $p'|_{E(\Gamma/[e]_{p'})}=p|_{E(\Gamma/[e]_{p'})}$ by the construction of $p'$, then $(\Gamma/[e]_{p'},p'|_{E(\Gamma/[e]_{p'})})$ is an enriched graph, and hence so is $(\Gamma,p')$.\par
  In the second case we have $e_1\notin[e]_p$, and then $[e]_p=[e]_{p'}$. Recall that $(\Gamma/[e]_p,q)$ is an enriched graph, where $q:=p|_{E(\Gamma/[e]_p)}$. Therefore $[e_1]_q=[e_1]_p$ and $[e_2]_q=[e_2]_p$ are consecutive equivalence classes in $(\Gamma/[e]_p,q)$, and hence, by the induction hypothesis, there exists an enriched graph  $(\Gamma/[e]_p,q')$ satisfying $e\<_{q'} e'$ if and only if $e\<_q e'$ or $\{e,e'\}\in[e_1]_q\cup[e_2]_q$. To conclude the proof, just note that $q'=p'|_{E(\Gamma/[e]_{p'})}$ (recall that $[e]_p=[e]_{p'}$), and $e\<_{p'} e'$ for every $e'\in E(\Gamma)$, hence $(\Gamma,p')$ is an enriched graph.
\end{proof}

\begin{Prop}
\label{prop:simple}
If the enriched graph $(\Gamma,p)$ specializes to $(\Gamma,p')$ then either $p=p'$ or there exists a simple specialization $(\Gamma,p)\leadsto(\Gamma,p_1)$ such that $(\Gamma,p_1)$ specializes to $(\Gamma,p')$. Moreover the specialization $(\Gamma,p)\leadsto (\Gamma,p')$ is simple if and only if $\rank(p')=\rank(p)-1$.
\end{Prop}
\begin{proof}
Let $E(\Gamma)_p\to E(\Gamma)_{p'}$ be the induced surjective map. Either this map is a bijection and hence $p=p'$, or there exist two different equivalence classes $[e_1]_p$ and $[e_2]_p$ with the same image, or equivalently $e_1\sim_{p'} e_2$. If we can find two such classes that are consecutive with respect to $p$ then applying the construction in Lemma-Definition \ref{lemdef:simple} we get the result. We prove that there exist such consecutive classes by induction on the number of edges of $\Gamma$. If $\Gamma$ has only one edge, or no edges, the result is trivial.\par
   If $\Gamma$ is not biconnected, then we have specializations 
   \[
   (\Gamma_i,p|_{E(\Gamma_i)})\leadsto(\Gamma_i,p'|_{E(\Gamma_i)}),
   \]
    for each one of its biconnected components $\Gamma_1,\ldots \Gamma_m$. By the induction hypothesis, either $p|_{E(\Gamma_i)}=p'|_{E(\Gamma_i)}$ for every $i$ and so $p=p'$ or there exists one such $i$ and a simple specialization $(\Gamma_i,p|_{E(\Gamma_i)})\leadsto(\Gamma_i,p_{i,1})$ such that $(\Gamma_i,p_{i,1})$ specializes to $(\Gamma_i,p'|_{E(\Gamma_i)})$. In this case just set $p_1$ as the enriched structure on $\Gamma$ such that $p_1|_{E(\Gamma_j)}=p|_{E(\Gamma_j)}$ for $j\neq i$ and $p_1|_{E(\Gamma_i)}=p_{i,1}$.\par
  If $\Gamma$ is biconnected, by Condition $(2')$ in Definition \ref{def:enrichedgraph} there exists an edge $e\in E(\Gamma)$ such that $e\<_p e'$ for every $e'\in E(\Gamma)$. We have two cases. In the first case there exists $[e_1]_p$ such that $e_1\sim_{p'} e$ and $e_1\nsim_p e$. We have $[e]_p<_p[e_1]_p$ and hence we can find $[e_2]_p$ such that $[e]_p<_p[e_2]_p\<_p[e_1]_p$ (and hence $e\sim_{p'} e_2$), with $[e]_p$ and $[e_2]_p$ consecutive.\par
   In the second case, we have $[e]_p=[e]_{p'}$, then both equivalence classes satisfy Condition $(2')$ in Definition \ref{def:enrichedgraph}. Hence $(\Gamma/[e]_p,p|_{E(\Gamma/[e]_p)})$ specializes to the enriched graph $(\Gamma/[e]_p,p'|_{E(\Gamma/[e]_p)})$ and, by the induction hypothesis, we can find two consecutive equivalence classes $[e_1]_p$ and $[e_2]_p$ with the same image, as required.\par
	If $(\Gamma,p)\leadsto (\Gamma,p')$ is simple then there exist only two classes $[e_1]_p$ and $[e_2]_p$ with the same image via the map $E(\Gamma)_p\to E(\Gamma)_{p'}$, hence $\rank(p')=\rank(p)-1$. Finally, the fact that any specialization is a composition of simple specializations implies that a specialization which lowers the rank by 1 is simple.
\end{proof}

\begin{Cor}
\label{cor:inclusion}
Given a specialization $i\col(\Gamma,p)\leadsto (\Gamma',p')$ there is a natural inclusion $E(\Gamma')_{p'}\hookrightarrow E(\Gamma)_p$, where the class $[e']_{p'}$ maps to the class of the minimal edge $e$ (with respect to $p$) such that $e\sim_{p'}e'$.
\end{Cor}
\begin{proof}
First we can reduce to the case where $\Gamma=\Gamma'$. Indeed, there is a lower set $S\subset E(\Gamma)$ with respect to $p$ such that $\Gamma'=\Gamma/S$ and hence we can factor $i$ as 
\[
(\Gamma,p)\leadsto (\Gamma',p|_{E(\Gamma')})\leadsto (\Gamma',p').
\]
Since $S$ is a union of equivalence classes in $E(\Gamma)_p$, we have a natural injection
\[
E(\Gamma')_{p|_{E(\Gamma')}}\hookrightarrow E(\Gamma)_p,
\]
because every equivalence class in $E(\Gamma')_{p|_{E(\Gamma')}}$ is an equivalence class in $E(\Gamma)_p$ which is not contained in $S$.\par
We will now proceed by induction. If $i$ is simple, then by Lemma-Definition \ref{lemdef:simple} there exist two consecutive classes $[e_1]_p$ and $[e_2]_p$ such that there is a bijection $E(\Gamma)_p\backslash\{[e_1]_p,[e_2]_p\}\to E(\Gamma)_{p'}\backslash\{[e_1]_{p'}\}$. Now to define the map $E(\Gamma)_{p'}\hookrightarrow E(\Gamma)_p$ one just uses this bijection and define the image of $[e_1]_{p'}$ as $[e_1]_p$ (the minimal of the consecutive classes).
   If $i$ is not simple, then by Proposition \ref{prop:simple} we can factor it through
\[
(\Gamma,p)\leadsto(\Gamma,p_1)\leadsto (\Gamma,p')
\]
where the first specialization is simple. By the induction hypothesis there exists a map $E(\Gamma)_{p'}\hookrightarrow E(\Gamma)_{p_1}$ and by the initial case we have the map $E(\Gamma)_{p_1}\hookrightarrow E(\Gamma)_p$.\par
  Clearly this map defined inductively satisfies the minimality condition in the statement. The proof is complete.
\end{proof}

\subsection{Graphs and their fans.}
\label{sec:fan}
 Given an enriched graph $(\Gamma,p)$, we define the cone $K(\Gamma,p)\subset \mathbb{R}^{E(\Gamma)}_{>0}$ as the cone satisfying the linear relations $x_e\leq x_{e'}$ when $e\<_p e'$ with equality if and only if $e\sim_p e'$. Here $(x_e)_{e\in E(\Gamma)}$ are the coordinates of $\mathbb{R}^{E(\Gamma)}$. It is easy to see that
\begin{equation}
\label{eq:dimrank}
\dim K(\Gamma,p)=\rank(p)
\end{equation}\par

The purpose of this subsection is to study the geometric properties of the cones $K(\Gamma,p)$. We begin by proving that we can restrict ourselves to the case in which $\Gamma$ is biconnected.

\begin{Lem}
\label{lem:prod}
If $\Gamma$ has biconnected components $\Gamma_1, \ldots, \Gamma_m$, then for every enriched structure $p$ on $\Gamma$ we have 
\[
K(\Gamma,p)=\prod_{i=1}^m K(\Gamma_i,p_i)
\]
where $p_i$ is the restriction of $p$ to $E(\Gamma_i)$.
\end{Lem}
\begin{proof}
We have a map 
\begin{eqnarray*}
K(\Gamma,p)&\lra&\prod_{i=1}^m \mathbb{R}^{E(\Gamma_i)}_{>0}\\
(x_e)_{e\in E(\Gamma)}&\longmapsto& \prod(x_e)_{e\in E(\Gamma_i)}.
\end{eqnarray*}
Clearly this map is injective. Since the preorder $p_i$ is the restriction of the preorder $p$ and $p$ only compares edges in the same set $E(\Gamma_i)$, the image of this map is
\[
\prod_{i=1}^mK(\Gamma_i,p_i).
\]
\end{proof}
\begin{Lem}
\label{lem:S}
Let $\Gamma$ be a biconnected graph and $(x_e)$ be a point of $\mathbb{R}_{>0}^{E(\Gamma)}$. Define the subset $S\subset E(\Gamma)$ of the elements $e\in E(\Gamma)$ such that $x_e$ is minimum, i.e., $x_e\leq x_{e'}$ for every $e'\in E(\Gamma)$. If $p$ is an enriched structure on $\Gamma$ and $(x_e)\in K(\Gamma,p)$, then $S$ satisfies Condition (2) in Definition \ref{def:enrichedgraph}. In particular, $S$ does not depend on the choice of the point $(x_e)$ inside $K(\Gamma,p)$.
\end{Lem}

\begin{proof}
Let $e$ be an edge such that $x_{e}$ is minimum. By Condition $(2')$ in Definition \ref{def:enrichedgraph}, there is an equivalence class $[e_0]_p$ such that $e_0\<_p e'$ for every $e'\in E(\Gamma)$. Hence we have $x_{e_0}\leq x_{e}$, by the definition of the cone $K(\Gamma,p)$. Since $x_{e}$ is minimum, we must have equality $x_{e}=x_{e_0}$ and then $e\sim_p e_0$, which implies $[e]_p=[e_0]_p$. Moreover, if $e'\sim_p e_0$, then we have $x_{e'}=x_e$. Hence $S=[e_0]_p$ and satisfies Condition (2) in Definition \ref{def:enrichedgraph}.
\end{proof}

Now, we begin to prove that the cones $K(\Gamma,p)$ will induce a fan. We start with the following proposition.

\begin{Prop} 
\label{prop:union}
Given a graph $\Gamma$ then $\mathbb{R}^{E(\Gamma)}_{>0}$ is the disjoint union of $K(\Gamma,p)$ where $p$ runs through all enriched structures on $\Gamma$.
\end{Prop}
\begin{proof}
The proof is by induction on the number of edges. If $\Gamma$ has only one edge, or no edges, the result is trivial. Assume now that $\Gamma$ has at least two edges.\par
  If $\Gamma$ is not biconnected, with biconnected components $\Gamma_1,\dots,\Gamma_m$, then $\mathbb{R}^{E(\Gamma)}_{>0}=\prod \mathbb{R}^{E(\Gamma_i)}_{>0}$. By the induction hypothesis we get
	\[
	\mathbb{R}^{E(\Gamma_i)}_{>0}=\coprod_{p_i} K(\Gamma_i,p_i),
	\]
	where the disjoint union is over every enriched structure $p_i$ on $\Gamma_i$. Hence, 
	\begin{equation}
	\label{eq:pi}
  \mathbb{R}^{E(\Gamma)}_{>0}=\coprod_{(p_1,\ldots, p_m)} \prod_{i=1}^m K(\Gamma_i,p_i)
	\end{equation}
	where the disjoint union is over all tuples $(p_i)$, $p_i$ an enriched structure on $\Gamma_i$. Since, by Definition \ref{def:enrichedgraph}, every enriched structure $p$ on $\Gamma$ is induced by its restrictions $p_i$ on $\Gamma_i$, and, conversely, every tuple $(p_i)$ induces an enriched structure $p$ on $\Gamma$, the disjoint union in Equation \eqref{eq:pi} is a disjoint union over every enriched structure $p$ on $\Gamma$, therefore we have
	\[
	\mathbb{R}^{E(\Gamma)}_{>0}=\coprod_p \prod_{i=1}^m K(\Gamma_i,p|_{E(\Gamma_i)})=\coprod_p K(\Gamma,p),
	\]
	where the last equation is the result of Lemma \ref{lem:prod}.\par	
	 If $\Gamma$ is biconnected, let us first prove that the cones $K(\Gamma,p)$ and $K(\Gamma,p')$ are disjoint for $p\neq p'$. Indeed, assume, by contradiction, that there is an element $(x_e)_{e\in E(\Gamma)}$ in both cones. \par
	Since $(x_e)$ belongs to both cones $K(\Gamma,p)$ and $K(\Gamma,p')$, defining $S$ as in Lemma \ref{lem:S}, we get that $S$ satisfies Condition (2) in Definition \ref{def:enrichedgraph} for both $(\Gamma,p)$ and $(\Gamma,p')$. Hence, we get that $(\Gamma/S,q)$ and $(\Gamma/S,q')$ are enriched graphs, where $q$ and $q'$ are the restrictions of $p$ and $p'$ to $\Gamma/S$. Moreover, the element $(x_{e'})_{e'\in E(\Gamma/S)}$ belongs to both cones $K(\Gamma/S,q)$ and $K(\Gamma/S,q')$. By the induction hypothesis, $q$ and $q'$ must be the same, hence $p=p'$, a contradiction.\par
	 All that is left to prove is that every element $(x_e)\in \mathbb{R}_{>0}^{E(\Gamma)}$ belongs to some cone $K(\Gamma,p)$. Define $S$ as in Lemma \ref{lem:S}. By the  induction hypothesis, there exists an enriched structure $q$ on $\Gamma/S$ such that the cone $K(\Gamma/S,q)$ contains the element $(x_{e'})_{e'\in E(\Gamma/S)}$. Then let $p$ be the preorder on $\Gamma$ defined as $e\<_p e'$ if and only if $e\in S$ or $e\<_q e'$. It is clear that $p$ is an enriched structure on $\Gamma$ and $(x_e)\in K(\Gamma,p)$.
\end{proof}

 We now consider the closure $\overline{K(\Gamma,p)}$ of $K(\Gamma,p)$ in $\mathbb{R}^{E(\Gamma)}$. By Proposition \ref{prop:union} it is expected that these closures form a fan. However, we are still missing some cones corresponding to the faces of $\overline{K(\Gamma,p)}$ contained in the hyperplanes $x_e=0$. These missing cones can be recovered by considering specializations of graphs. The following observation together with the following two propositions address this problem.\par
   Clearly the cone $\overline{K(\Gamma,p)}$ is given by the linear relations $0\leq x_e\leq x_{e'}$ if $e\<_p e'$. Given a specialization $i\colon (\Gamma,p)\leadsto (\Gamma',p')$ there exists an integral linear injective map 
\begin{equation}
\label{eq:i}
\overline{i}\colon \mathbb{R}^{E(\Gamma')}\to\mathbb{R}^{E(\Gamma)}
\end{equation}
induced by the natural inclusion $E(\Gamma')\subset E(\Gamma)$. Note that $\overline{i}$ only depends on the specialization of graphs $\Gamma\leadsto \Gamma'$. In particular, when $\Gamma=\Gamma'$, then $\overline{i}$ is the identity.

\begin{Prop}
\label{prop:unionK}
Given an enriched graph $(\Gamma,p)$ the cone $\overline{K(\Gamma,p)}$ is the disjoint union $\coprod\overline{i}(K(\Gamma',p'))$ for $(\Gamma',p')$ running through all specializations $i\colon(\Gamma,p)\leadsto (\Gamma',p')$.
\end{Prop}
\begin{proof}
The proof is by induction on the number of edges. If $\Gamma$ has only one edge, or no edges, the result is trivial. Assume now that $\Gamma$ has at least two edges.\par
  If $\Gamma$ is not biconnected, let $\Gamma_1,\ldots, \Gamma_m$ be its biconnected components. Then, by Lemma \ref{lem:prod}, $\overline{K(\Gamma,p)}=\prod \overline{K(\Gamma_j,p_j)}$, where $p_j=p|_{E(\Gamma_j)}$. Hence, by the induction hypothesis, we have
	\[
	\overline{K(\Gamma_j,p_j)}=\coprod \overline{i_j}(\Gamma_j',p_j')
	\]
	where the union runs through all specializations $i_j\colon(\Gamma_i,p_i)\leadsto(\Gamma_j',p_j')$. Hence
	\[
  \overline{K(\Gamma,p)}=\coprod_{i_1,\ldots, i_m} \prod_{j=1}^m \overline{i_j}(K(\Gamma_j',p_j'))=\coprod_{i} \overline{i}(K(\Gamma',p')).
	\]
	where $i\colon(\Gamma,p)\leadsto (\Gamma',p')$ is such that $i|_{\Gamma_j}=i_j$.\par
	
	 If $\Gamma$ is biconnected, we first prove that the cones $\overline{i_1}(K(\Gamma_1,p_1))$ and $\overline{i_2}(K(\Gamma_2,p_2))$ are disjoint, where $i_1\colon (\Gamma,p)\leadsto (\Gamma_1,p_1)$ and $i_2\colon (\Gamma,p)\leadsto (\Gamma_2,p_2)$ are distinct specializations.\par
	First we note that we can assume that $\Gamma_1=\Gamma_2$. Indeed, if there is an edge $e\in E(\Gamma_1)\backslash E(\Gamma_2)$ then $\overline{i_1}(K(\Gamma_1,p_1)$ is contained in $x_e>0$ and $\overline{i_2}(K(\Gamma_2,p_2))$ is contained in $x_e=0$ and hence we are done.\par
	Since $\Gamma_1=\Gamma_2$, using Proposition \ref{prop:union} we have that the cones $K(\Gamma_1,p_1)$ and $K(\Gamma_1,p_2)$ are disjoint. Moreover, we have that $\overline{i_1}=\overline{i_2}$, and since $\overline{i_1}$ is injective it follows that $\overline{i_1}(K(\Gamma_1,p_1))$ and $\overline{i_2}(K(\Gamma_2,p_2))$ are disjoint.\par
	
	All that is left to prove is that every element $(x_e)\in \overline{K(\Gamma,p)}$ belongs to some cone $\overline{i}(K(\Gamma',p'))$. Fix $(x_e)\in\overline{K(\Gamma,p)}$ and define $S$ as in Lemma \ref{lem:S}. By Lemma \ref{lem:S},  we get that $S$ is a lower set for $(\Gamma,p)$. We have two cases.\par
	In the first case $x_e=0$ for every $e\in S$. Then, by the induction hypothesis, there exists a specialization $i'\colon (\Gamma/S,p|_{E(\Gamma/S)})\leadsto (\Gamma',p')$ such that the cone $\overline{i'}(K(\Gamma',p'))$ contains the element $(x_{e'})_{e'\in E(\Gamma/S)}$. Then set $i$ as the composition of $(\Gamma,p)\leadsto (\Gamma/S,p|_{E(\Gamma/S)})$ and $i'\colon(\Gamma/S,p|_{E(\Gamma/S)})\leadsto (\Gamma',p')$. Clearly $(x_e)\in \overline{i}(K(\Gamma',p'))$. \par
	In the second case, no coordinate $x_e$ is zero, for every $e\in S$. Then $(x_e)$ is contained in some cone $K(\Gamma,p')$ by Proposition \ref{prop:union}. All that we have to prove is that $(\Gamma,p)\leadsto (\Gamma,p')$. By Definition \ref{def:specializes} it is enough to show that if $e\<_p e'$ then $e\<_{p'} e'$. By Lemma \ref{lem:S}, we have that if $e\in S$ then $e\<_{p'} e'$ for every $e'\in E(\Gamma)$. Since $(x_e)_{e\in E(\Gamma/S)}$ is contained in $K(\Gamma/S,p'|_{E(\Gamma/S)})$ and in $\overline{K(\Gamma/S,p|_{E(\Gamma/S)})}$, by the induction hypothesis we get a specialization $(\Gamma/S,p|_{E(\Gamma/S)})\leadsto (\Gamma/S,p'|_{E(\Gamma/S)})$. This implies that if $e\notin S$ and $e\<_pe'$ (and therefore $e'\notin S$) then $e\<_{p'} e'$. On the other hand if $e\<_p e'$ and $e\in S$, then $e\<_{p'}e'$ by the observation above. This finishes the proof.	\end{proof}

\begin{Prop} 
\label{prop:face}
If $i\colon(\Gamma,p)\leadsto(\Gamma',p')$ is a specialization, then $\overline{i}(\overline{K(\Gamma',p')})$ is a face of the cone $\overline{K(\Gamma,p)}$. Conversely, any face of the cone $\overline{K(\Gamma,p)}$ is of the form $\overline{i}(\overline{K(\Gamma',p')})$ for some specialization $i\colon (\Gamma,p)\leadsto (\Gamma',p')$.
\end{Prop}
\begin{proof}
Clearly the image of $\overline{K(\Gamma',p')}$ is contained in $\overline{K(\Gamma,p)}$ by Condition (3) in Definition \ref{def:specializes}. 
To prove that $\overline{i}(\overline{K(\Gamma',p')})$ is a face of $\overline{K(\Gamma,p)}$ it is enough to prove that $\overline{i}(\overline{K(\Gamma',p')})$ is the intersection of $\overline{K(\Gamma,p)}$ with some integral linear space of codimension one such that the cone $\overline{K(\Gamma,p)}$ lies in one of the two half-spaces determined by such a linear space.\par
   First, we show that we can reduce to the case where $\Gamma=\Gamma'$ (and hence $\overline{i}$ is the identity). Indeed, by Definition \ref{def:specializes} there is a lower set (with respect to $p$) $S\subset E(\Gamma)$ such that $\Gamma'=\Gamma/S$ and we can factor $i$ as $(\Gamma,p)\leadsto (\Gamma/S,p|_{\Gamma/S})\leadsto (\Gamma/S,p')$. Then $\overline{i}^{-1}(\overline{K(\Gamma,p)})=\overline{K(\Gamma/S,p|_{\Gamma/S})}$, and the image  $\overline{i}(\overline{K(\Gamma/S,p|_{E(\Gamma/S)})})$ is the intersection of the cone $\overline{K(\Gamma,p)}$ with the integral linear subspace of codimension one
\[
H:\sum_{e\in S} x_e=0.
\]
Clearly $\overline{K(\Gamma,p)}$ lies in one of the two half-spaces determined by $H$, and hence $\overline{i}(K(\Gamma/S,p|_{E(\Gamma/S)})$ is a face of $\overline{K(\Gamma,p)}$.
Therefore it suffices to prove that $\overline{K(\Gamma',p')}$ is a face of $\overline{K(\Gamma/S,p|_{\Gamma/S})}$.\par
  We now note that, by Proposition \ref{prop:simple}, we can assume that the specialization $(\Gamma,p)\leadsto(\Gamma,p')$ is simple. If the specialization is simple, by Lemma-Definition \ref{lemdef:simple} there exist consecutive classes $[e_1]_p$ and $[e_2]_p$ such that $e\<_{p'} e'$ if and only if $e\<_p e'$ or $\{e,e'\}\subset [e_1]_p\cup [e_2]_p$. In particular we have $e_1\sim_{p'} e_2$. We now prove that $\overline{K(\Gamma,p')}$ is the intersection of $\overline{K(\Gamma,p)}$ with the integral linear subspace of codimension one
\[
H: x_{e_2}-x_{e_1}=0.
\]
In fact, the equations of $\overline{K(\Gamma,p')}$ are $x_e\leq x_{e'}$ when $e\<_{p'} e'$, while the equations of $\overline{K(\Gamma,p)}$ are $x_e\leq x_{e'}$ when $e\<_p e'$. Note that $e\<_{p'} e'$ if and only if $e\<_p e'$ or $\{e,e'\}\subset [e_1]_p\cup[e_2]_p$. However, in the latter case, we have $e\<_p e'$ unless $e\in [e_2]_p$ and $e'\in [e_1]_p$. Therefore we have that the equations of $\overline{K(\Gamma,p')}$ are the equations of $\overline{K(\Gamma,p)}$ together with the equations $x_e\leq x_{e'}$, when $e'\sim_p e_1$ and $e\sim_p e_2$. All that is left to prove is that the latter equations arise from the addition of the equation of $H$. Indeed, since $e'\sim_p e_1$ and $e\sim_p e_2$, we have that $e\sim_{p'}e'$, hence the equations $x_e\leq x_{e'}$ and $x_{e'}\leq x_e$ become $x_e=x_{e'}$. However, modulo the equations $x_{e}=x_{e_2}$ and $x_{e'}=x_{e_1}$ (which are contained in the set of equations defining $\overline{K(\Gamma,p)}$, because $e\sim_p e_2$ and $e'\sim_p e_1$), we have that the equation $x_e=x_{e'}$ is equivalent to the equation $x_{e_1}=x_{e_2}$, where the latter is the equation of $H$. This proves that the equations of $\overline{K(\Gamma,p)}$  together with the equation of $H$ are equations defining $\overline{K(\Gamma,p')}$.\par
  To conclude the proof of the first statement of the proposition, just note that $\overline{K(\Gamma,p)}$ is in one of the two half-spaces determined by $H$, because, being $[e_1]_p$ and $[e_2]_p$ consecutive, we have that $x_{e_1}\leq x_{e_2}$ is an equation of $\overline{K(\Gamma,p)}$.\par
	The second statement follows from Proposition \ref{prop:unionK} and the first statement.
\end{proof}
 
The following corollary is not necessary to define the fan associated to a graph, but will be used in the proof of Proposition \ref{prop:cod1}.
\begin{Cor}
\label{cor:generic}
Every enriched structure is a specialization of a generic one.
\end{Cor}
\begin{proof}
Let $(\Gamma,p)$ be an enriched graph, with $p$ nongeneric. Then, by Proposition \ref{prop:union}, the cone $K(\Gamma,p)$ is in the closure of $K(\Gamma,p')$ for $p'$ generic (remember that the dimension of $K(\Gamma,p)$ is the rank of $p$). Hence the cone $K(\Gamma,p)$ intersects the relative interior of some of the faces of $\overline{K(\Gamma,p')}$. By Proposition \ref{prop:face}, all of these faces are of the form $\overline{i}(\overline{K(\Gamma_j,p_j)})$ where $(\Gamma,p')\leadsto(\Gamma_j,p_j)$ runs through all specializations of $(\Gamma,p')$. This means that there exists a specialization $(\Gamma,p')\leadsto(\Gamma,\tilde{p})$ such that $K(\Gamma,p)\cap K(\Gamma,\tilde{p})\neq\emptyset$, which implies, again by Proposition \ref{prop:union}, that $(\Gamma,p)=(\Gamma,\tilde{p})$.
\end{proof}

Finally, we can define the fan associated to a graph.

\begin{Def}
\label{def:fan}
Given a graph $\Gamma$ we define the fan $\Sigma_\Gamma$ as the collection of cones $\overline{i}(\overline{K(\Gamma',p')})$, where $i$ runs through all specializations $i\colon (\Gamma,p) \leadsto (\Gamma',p')$ of enriched graphs and the closure is taken in $\mathbb{R}^{E(\Gamma)}$. Propositions \ref{prop:union}, \ref{prop:unionK} and \ref{prop:face} assure that $\Sigma_\Gamma$ is in fact a fan. Moreover, given a specialization $i\col(\Gamma,p)\leadsto(\Gamma',p')$, we have that 
\[
\Sigma_{\Gamma'}=\{\overline{i}^{-1}(\sigma) | \sigma\in \Sigma_\Gamma\}.
\]
\end{Def}

\begin{Exa}
\label{ex:fans}
Following Examples \ref{ex:circular} and \ref{ex:2c} one can draw (sections of) the decompositions of the fans in the case of circular graphs and graphs with only two vertices.  In Figure \ref{fig:circular} we consider the case of a circular graph with 3 edges. In this case, we have 6 generic enriched structures each one of which corresponds to a permutation of the set of edges and giving rise to a maximal cone in the associated fan. In Figure \ref{fig:2c} we consider the case of a graph with 3 edges and 2 vertices. In this case, we have 3 generic enriched structures each one of which corresponds to an edge and giving rise to a maximal cone in the associated fan. 
\begin{figure}[htb]
\begin{minipage}{0.48\linewidth}
\[
\begin{xy} <15pt,0pt>:
(-4,0)*{\scriptstyle\bullet}="a"; 
(4,0)*{\scriptstyle\bullet}="b";
(0,6)*{\scriptstyle\bullet}="c";
(0,2)*{\scriptstyle\bullet}="v";
(-2,3)*{\scriptstyle\bullet}="d";
(2,3)*{\scriptstyle\bullet}="e";
(0,0)*{\scriptstyle\bullet}="f";
{\ar@{-}"a"*{};"b"*{}};
{\ar@{-}"a"*{};"c"*{}};
{\ar@{-}"a"*{};"v"*{}};
{\ar@{-}"c"*{};"b"*{}};
{\ar@{-}"v"*{};"b"*{}};
{\ar@{-}"v"*{};"c"*{}};
{\ar@{-}"v"*{};"d"*{}};
{\ar@{-}"v"*{};"e"*{}};
{\ar@{-}"v"*{};"f"*{}};
\end{xy}
\]
\caption{The decomposition of the fan of a circular graph.}
\label{fig:circular}
\end{minipage}\quad
\begin{minipage}{0.48\linewidth}
\[
\begin{xy} <15pt,0pt>:
(-4,0)*{\scriptstyle\bullet}="a"; 
(4,0)*{\scriptstyle\bullet}="b";
(0,6)*{\scriptstyle\bullet}="c";
(0,2)*{\scriptstyle\bullet}="v";
{\ar@{-}"a"*{};"b"*{}};
{\ar@{-}"a"*{};"c"*{}};
{\ar@{-}"a"*{};"v"*{}};
{\ar@{-}"c"*{};"b"*{}};
{\ar@{-}"v"*{};"b"*{}};
{\ar@{-}"v"*{};"c"*{}};
\end{xy}
\]
\caption{The decomposition of the fan of a graph with 2 vertices.}
\label{fig:2c}
\end{minipage}
\end{figure}
\end{Exa}

 There is an alternative description of $\overline{K(\Gamma,p)}$ in terms of its ray generators. Recall the notation introduced in Equation \eqref{eq:RE}. 
\begin{Prop}
\label{prop:generators}
The ray generators of the cone $\overline{K(\Gamma,p)}$ are the vectors 
\[
v_T:=\sum_{e\in T}e
\]
where $T$ runs through all irreducible upper sets of $E(\Gamma)$ with respect to $p$.
\end{Prop}
\begin{proof}
The proof is by induction on the number of edges. If $\Gamma$ has only one edge, or no edges, the result is trivial. Assume now that $\Gamma$ has at least two edges.\par
  If $\Gamma$ is not biconnected then, by Lemma \ref{lem:prod}, $\overline{K(\Gamma,p)}=\prod \overline{K(\Gamma_i,p_i)}$, where $\Gamma_i$ are the biconnected components of $\Gamma$ and $p_i:=p|_{E(\Gamma_i)}$. The result follows by the induction hypothesis and from the fact that an irreducible upper set is contained in $E(\Gamma_i)$ for some $i$.\par
	If $\Gamma$ is biconnected, by Condition $(2)$ in Definition \ref{def:enrichedgraph}, there exists a lower set $S\subset E(\Gamma)$ such that $e\<_p e'$ for every $e\in S$ and $e'\in E(\Gamma)$ and $(\Gamma/S,p|_{E(\Gamma/S)})$ is an enriched graph. Then, an upper set of $E(\Gamma)$ is either an upper set of $E(\Gamma/S)$ or it is the entire set $E(\Gamma)$. By the induction hypothesis the cone $\overline{K(\Gamma/S,p|_{E(\Gamma/S)})}$ is generated by the vectors $v_T$ where $T$ runs through all irreducible upper sets of $E(\Gamma/S)$ with respect to $p|_{E(\Gamma/S)}$. Let $i$ be the specialization $(\Gamma,p)\leadsto (\Gamma/S,p|_{E(\Gamma/S)})$. We now have to prove that 
\begin{equation}
\label{eq:iv}
\overline{K(\Gamma,p)}=\text{cone}\left(\overline{i}\left(\overline{K(\Gamma/S,p|_{E(\Gamma/S)})}\right), v_{E(\Gamma)}\right).
\end{equation}
The right hand side is clearly contained in the left hand side, because all vectors in the right hand side satisfy the equations defining the cone $\overline{K(\Gamma,p)}$. Now, given a vector $v=(x_e)_{e\in E(\Gamma)}$ in $\overline{K(\Gamma,p)}$, we have $x_e=x_{e'}=:\alpha$ for every $e,e'\in S$, hence $v-\alpha v_{E(\Gamma)}$ is in the image $\overline{i}(\overline{K(\Gamma/S,p|_{E(\Gamma/S)})})$ which finishes the proof.
\end{proof}

\begin{Prop}
\label{prop:smooth}
The cone $\overline{K(\Gamma,p)}$ is smooth, i.e., the vectors $v_T$ can be completed to form a basis for $\mathbb{Z}^{E(\Gamma)}$.
\end{Prop}
\begin{proof}
  First, we claim that $T\subset E(\Gamma)$ is an irreducible upper set with respect to $p$ if and only if $T=T_e:=\{e'|e\<_pe'\}$ for some $e\in E(\Gamma)$. Indeed, the set $T_e$ is an upper set and it is irreducible, because if $T_e=T_1\cup T_2$ with $T_1,T_2$ upper sets, then $e$ must belong to $T_i$, for some $i=1,2$, and in this case $T_e=T_i$. Conversely, if $T$ is an irreducible upper set, let $e_1,\ldots e_m$ be its minimal elements. Then $T=\bigcup_{i=1}^m T_{e_i}$, and since $T$ is irreducible, it follows that $T=T_{e_i}$ for some $i=1,\ldots,m$.\par
  Since, by Corollary \ref{cor:generic}, every cone $\overline{K(\Gamma,p)}$ is a face of a cone $\overline{K(\Gamma,p')}$ with $p'$ a generic enriched structure on $\Gamma$, we can assume without loss of generality that $p$ is generic, which means that $p$ is a partial order, and in this case all the $T_e$'s are different.\par
  For an edge $e\in E(\Gamma)$, let $\{e_1,\ldots e_m\}$ be the (possibly empty) set of all the edges in $E(\Gamma)$ such that $e$ and $e_i$ are consecutive for every $i$. Since $p$ is a partial order, we have $T_e\setminus \bigcup_{i=1}^m T_{e_i}=\{e\}$. Moreover, since by Corollary \ref{cor:hasse} the Hasse diagram of $(E(\Gamma),p)$ is a disjoint union of rooted trees (one for each biconnected component), we also have that the $T_{e_i}$'s are disjoint. It follows that $e=v_{T_e}-\sum_{i=1}^m v_{T_{e_i}}$ (recall the definition of $v_T$ in Proposition \ref{prop:generators}). This proves that the $v_{T_e}$'s generate $\mathbb{Z}^{E(\Gamma)}$ and, since there are exactly $|E(\Gamma)|$ of them, they must form a base.
\end{proof}

We finish this section by giving a description of how one can obtain $\Sigma_\Gamma$ as a sequence of star subdivisions of the cone $\mathbb{R}^{E(\Gamma)}_{\geq0}$. The result will later translate to Theorem \ref{thm:mainEX}.

Given a graph $\Gamma$, we say that a sequence $(i_j\col \Gamma\leadsto \Gamma'_j)_{j=1,\dots,n}$ of specializations of $\Gamma$ is \emph{good} if
\begin{enumerate}
\item each $\Gamma'_j$ is biconnected with at least 2 edges;
\item  every specialization $\Gamma\leadsto \Gamma'$, with $\Gamma'$ biconnected, appears exactly once in the sequence;
\item if $|E(\Gamma'_{j_1})|<|E(\Gamma'_{j_2})|$, then $j_1>j_2$.
\end{enumerate}

\begin{Prop}
\label{prop:star}
Let $\Gamma$ be a graph and $(i_j\colon \Gamma\leadsto \Gamma'_j)_{j=1,\dots,n}$ be a (possibly empty) good sequence of specializations of $\Gamma$. If we define the fans in $\mathbb{R}^{E(\Gamma)}$
\[
\Sigma_{0,\Gamma}:=\{\text{cone}(A)|A\subset E(\Gamma)\}\quad\text{and}\quad \Sigma_{j+1,\Gamma}:=\Sigma_{j,\Gamma}^\star(\tau_{j+1})
\]
where $\tau_j:=\text{cone}(E(\Gamma'_j))$, then $\Sigma_\Gamma=\Sigma_{n,\Gamma}$.
\end{Prop}
\begin{proof}
Since the proof is quite involved, we begin with an example. Let $\Gamma$ be the circular graph with 3 vertices and edges $e_1, e_2, e_3$. To avoid cumbersome notations we define $\Gamma_{e_j}:=\Gamma/\{e_j\}$ and, given a cone $\sigma$, we define the fan $\Sigma_\sigma:=\{\tau|\tau\prec\sigma\}$. Then, we have (see Figure \ref{fig:sigma0})
\[
\Sigma_{0,\Gamma}=\Sigma_{\text{cone}(\{e_1,e_2,e_3\})}.
\]
A good sequence for $\Gamma$ is given by the specializations $i_1\col\Gamma\leadsto\Gamma$, $i_2\col\Gamma\leadsto\Gamma_{e_1}$, $i_3\col\Gamma\leadsto\Gamma_{e_2}$, $i_4\col\Gamma\leadsto\Gamma_{e_3}$. Hence $\tau_1=\text{cone}(e_1,e_2,e_3)$, and let $v_0:=e_1+e_2+e_3$. By Definition \ref{def:star}, we see that (see Figure \ref{fig:sigma1})
\[
\Sigma_{1,\Gamma}=\Sigma_{\text{cone}(e_1,e_2,v_0)}\cup\Sigma_{\text{cone}(e_1,e_3,v_0)}\cup\Sigma_{\text{cone}(e_2,e_3,v_0)}.
\]
Now, $\tau_2=\text{cone}(E(\Gamma_{e_1}))=\text{cone}(e_2,e_3)$ and let $v_1:=e_2+e_3$. Then, by Definition \ref{def:star} we have 
\[
\Sigma_{2,\Gamma}=\{\sigma\in \Sigma_{1,\Gamma}|\tau_2\not\subset\sigma\}\cup\bigcup_{\tau_2\subset\sigma}\Sigma^\star_\sigma(\tau_2).
\]
\begin{figure}[htb]
\begin{minipage}{0.48\linewidth}
\[
\begin{xy} <15pt,0pt>:
(-4,0)*{\scriptstyle\bullet}="a"; 
(4,0)*{\scriptstyle\bullet}="b";
(0,6)*{\scriptstyle\bullet}="c";
{\ar@{-}"a"*{};"b"*{}};
{\ar@{-}"a"*{};"c"*{}};
{\ar@{-}"c"*{};"b"*{}};
"a"+(-0.25,-0.25)*{e_2};
"b"+(0.4,-0.25)*{e_1};
"c"+(0.25,0.25)*{e_3};
"c"+(-3.5,-3)*{\Sigma_{0,\Gamma_{e_1}}};
"c"+(3.5,-3)*{\Sigma_{0,\Gamma_{e_2}}};
"a"+(4,-0.5)*{\Sigma_{0,\Gamma_{e_3}}};
\end{xy}
\]
\caption{A section of the fan $\Sigma_{0,\Gamma}$.}
\label{fig:sigma0}
\end{minipage}\quad
\begin{minipage}{0.48\linewidth}
\[
\begin{xy} <15pt,0pt>:
(-4,0)*{\scriptstyle\bullet}="a"; 
(4,0)*{\scriptstyle\bullet}="b";
(0,6)*{\scriptstyle\bullet}="c";
(0,2)*{\scriptstyle\bullet}="v";
{\ar@{-}"a"*{};"b"*{}};
{\ar@{-}"a"*{};"c"*{}};
{\ar@{-}"a"*{};"v"*{}};
{\ar@{-}"c"*{};"b"*{}};
{\ar@{-}"v"*{};"b"*{}};
{\ar@{-}"v"*{};"c"*{}};
"v"+(0.35,0.25)*{v_0};
"a"+(-0.25,-0.25)*{e_2};
"b"+(0.4,-0.25)*{e_1};
"c"+(0.25,0.25)*{e_3};
"c"+(-3.5,-3)*{\Sigma_{0,\Gamma_{e_1}}};
"c"+(3.5,-3)*{\Sigma_{0,\Gamma_{e_2}}};
"a"+(4,-0.5)*{\Sigma_{0,\Gamma_{e_3}}};
\end{xy}
\]
\caption{A section of the fan $\Sigma_{1,\Gamma}$.}
\label{fig:sigma1}
\end{minipage}
\end{figure}

\noindent The first set is
\[
\{\sigma\in \Sigma_{1,\Gamma}|\tau_2\not\subset\sigma\}=\Sigma_{\text{cone}(e_1,e_2,v_0)}\cup\Sigma_{\text{cone}(e_1,e_3,v_0)}.
\]
and the second set is $\Sigma^\star_{\tau_2}(\tau_2)\cup\Sigma^\star_{\text{cone}(e_2,e_3,v_0)}(\tau_2)$, because $\tau_2$ and $\text{cone}(e_2,e_3,v_0)$ are the only cones in $\Sigma_{1,\Gamma}$ that contain $\tau_2$. By Definition \ref{def:star} we have $\Sigma^\star_{\tau_2}(\tau_2)\subset \Sigma^\star_{\text{cone}(e_2,e_3,v_0)}(\tau_2)$ and 
\[
\Sigma^\star_{\text{cone}(e_2,e_3,v_0)}(\tau_2)=\Sigma_{\text{cone}(e_2,v_0,v_1)}\cup\Sigma_{\text{cone}(e_3,v_0,v_1)},
\]
then (see Figure \ref{fig:sigma2})
\[
\Sigma_{2,\Gamma}=\Sigma_{\text{cone}(e_1,e_2,v_0)}\cup\Sigma_{\text{cone}(e_1,e_3,v_0)}\cup\Sigma_{\text{cone}(e_2,v_0,v_1)}\cup\Sigma_{\text{cone}(e_3,v_0,v_1)}.
\]
\begin{figure}[htb]
\begin{minipage}{0.48\linewidth}
\[
\begin{xy} <15pt,0pt>:
(-4,0)*{\scriptstyle\bullet}="a"; 
(4,0)*{\scriptstyle\bullet}="b";
(0,6)*{\scriptstyle\bullet}="c";
(0,2)*{\scriptstyle\bullet}="v";
(-2,3)*{\scriptstyle\bullet}="d";
{\ar@{-}"a"*{};"b"*{}};
{\ar@{-}"a"*{};"c"*{}};
{\ar@{-}"a"*{};"v"*{}};
{\ar@{-}"c"*{};"b"*{}};
{\ar@{-}"v"*{};"b"*{}};
{\ar@{-}"v"*{};"c"*{}};
{\ar@{-}"v"*{};"d"*{}};
"v"+(0.35,0.25)*{v_0};
"a"+(-0.25,-0.25)*{e_2};
"b"+(0.4,-0.25)*{e_1};
"c"+(0.25,0.25)*{e_3};
"d"+(-0.4,0.25)*{v_1};
"c"+(-4.5,-2.5)*{\Sigma_{1,\Gamma_{e_1}}};
"c"+(3.5,-3)*{\Sigma_{0,\Gamma_{e_2}}};
"a"+(4,-0.5)*{\Sigma_{0,\Gamma_{e_3}}};
\end{xy}
\]
\caption{A section of the fan $\Sigma_{2,\Gamma}$.}
\label{fig:sigma2}
\end{minipage}
\begin{minipage}{0.48\linewidth}
\[
\begin{xy} <15pt,0pt>:
(-4,0)*{\scriptstyle\bullet}="a"; 
(4,0)*{\scriptstyle\bullet}="b";
(0,6)*{\scriptstyle\bullet}="c";
(0,2)*{\scriptstyle\bullet}="v";
(-2,3)*{\scriptstyle\bullet}="d";
(2,3)*{\scriptstyle\bullet}="e";
{\ar@{-}"a"*{};"b"*{}};
{\ar@{-}"a"*{};"c"*{}};
{\ar@{-}"a"*{};"v"*{}};
{\ar@{-}"c"*{};"b"*{}};
{\ar@{-}"v"*{};"b"*{}};
{\ar@{-}"v"*{};"c"*{}};
{\ar@{-}"v"*{};"e"*{}};
{\ar@{-}"v"*{};"d"*{}};
"v"+(0,-0.35)*{v_0};
"a"+(-0.25,-0.25)*{e_2};
"b"+(0.4,-0.25)*{e_1};
"c"+(0.25,0.25)*{e_3};
"d"+(-0.4,0.25)*{v_1};
"e"+(0.4,0.25)*{v_2};
"c"+(-4.5,-2.5)*{\Sigma_{1,\Gamma_{e_1}}};
"c"+(4.5,-2.5)*{\Sigma_{1,\Gamma_{e_2}}};
"a"+(4,-0.5)*{\Sigma_{0,\Gamma_{e_3}}};
\end{xy}
\]
\caption{A section of the fan $\Sigma_{3,\Gamma}$.}
\label{fig:sigma3}
\end{minipage}
\end{figure}

\noindent Similarly, we can see that (see Figure \ref{fig:sigma3})
\[
\Sigma_{3,\Gamma}= \Sigma_{\text{cone}(e_1,e_2,v_0)}\cup\Sigma_{\text{cone}(e_1,v_0,v_2)}\cup\Sigma_{\text{cone}(e_3,v_0,v_2)}\cup\Sigma_{\text{cone}(e_2,v_0,v_1)}\cup\Sigma_{\text{cone}(e_3,v_0,v_1)},
\]
where $v_2:=e_1+e_3$. Finally (see Figure \ref{fig:sigma4}),
\begin{align*}
\Sigma_{4,\Gamma}=& \Sigma_{\text{cone}(e_1,v_0,v_3)}\cup\Sigma_{\text{cone}(e_2,v_0,v_3)}\cup\Sigma_{\text{cone}(e_1,v_0,v_2)}\cup\Sigma_{\text{cone}(e_3,v_0,v_2)}\cup\\
&\Sigma_{\text{cone}(e_2,v_0,v_1)}\cup\Sigma_{\text{cone}(e_3,v_0,v_1)},
\end{align*}
where $v_3:=e_1+e_2$. We note that $\Sigma_{4,\Gamma}=\Sigma_\Gamma$ (see Figures \ref{fig:sigma4} and \ref{fig:circular}), which proves the statement of the theorem in the example.

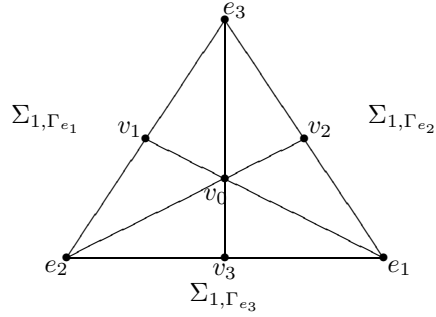
\begin{figure}[htb]
\begin{minipage}{0.48\linewidth}
\[
\begin{xy} <15pt,0pt>:
(-4,0)*{\scriptstyle\bullet}="a"; 
(4,0)*{\scriptstyle\bullet}="b";
(0,6)*{\scriptstyle\bullet}="c";
(0,2)*{\scriptstyle\bullet}="v";
(-2,3)*{\scriptstyle\bullet}="d";
(2,3)*{\scriptstyle\bullet}="e";
(0,0)*{\scriptstyle\bullet}="f";
{\ar@{-}"a"*{};"b"*{}};
{\ar@{-}"a"*{};"c"*{}};
{\ar@{-}"a"*{};"v"*{}};
{\ar@{-}"c"*{};"b"*{}};
{\ar@{-}"v"*{};"b"*{}};
{\ar@{-}"v"*{};"c"*{}};
{\ar@{-}"v"*{};"e"*{}};
{\ar@{-}"v"*{};"d"*{}};
{\ar@{-}"v"*{};"f"*{}};
"v"+(-0.23,-0.45)*{v_0};
"a"+(-0.25,-0.25)*{e_2};
"b"+(0.4,-0.25)*{e_1};
"c"+(0.25,0.25)*{e_3};
"d"+(-0.4,0.25)*{v_1};
"e"+(0.4,0.25)*{v_2};
"f"+(0,-0.35)*{v_3};
"c"+(-4.5,-2.5)*{\Sigma_{1,\Gamma_{e_1}}};
"c"+(4.5,-2.5)*{\Sigma_{1,\Gamma_{e_2}}};
"a"+(4,-1)*{\Sigma_{1,\Gamma_{e_3}}};
\end{xy}
\]
\caption{A section of the fan $\Sigma_{4,\Gamma}$.}
\label{fig:sigma4}
\end{minipage}
\end{figure}

    The proof in the general case will be by induction. Let us make an observation that explains how the induction step appears in the previous example. First note that one can regard the fans $\Sigma_{0,\Gamma_{e_j}}$, for $j=1,2,3$, as fans contained in $\Sigma_{0,\Gamma}$ (see Figure \ref{fig:sigma0}), and by the observation in Definition \ref{def:fan}, one can regard the fans $\Sigma_{\Gamma_{e_j}}$ as fans contained in $\Sigma_\Gamma$ as well. We note that the good sequence $(i_k)$ induces a good sequence for each $\Gamma_{e_j}$, which, in the case of the example, is the sequence with the unique term $i'_{j+1}\col\Gamma_{e_j}\leadsto\Gamma_{e_j}$. One can relate the construction of $\Sigma_{4,\Gamma}$ with the constructions of $\Sigma_{1,\Gamma_{e_j}}$, for $j=1,2,3$, as follows. We note that (recall Equation \eqref{eq:fanv})
\[
\Sigma_{1,\Gamma}=\text{fan}(\Sigma_{0,\Gamma_{e_1}},v_0)\cup\text{fan}(\Sigma_{0,\Gamma_{e_2}},v_0)\cup\text{fan}(\Sigma_{0,\Gamma_{e_3}},v_0).
\]
Now the star subdivisions $\Sigma_{2,\Gamma}, \Sigma_{3,\Gamma},\Sigma_{4,\Gamma}$ of $\Sigma_{1,\Gamma}$, will be induced by star subdivisions of $\Sigma_{0,\Gamma_{e_j}}$, i.e. (see Figures \ref{fig:sigma2}, \ref{fig:sigma3} and \ref{fig:sigma4}), 
\begin{align*}
\Sigma_{2,\Gamma}=&\text{fan}(\Sigma_{1,\Gamma_{e_1}},v_0)\cup\text{fan}(\Sigma_{0,\Gamma_{e_2}},v_0)\cup\text{fan}(\Sigma_{0,\Gamma_{e_3}},v_0),\\
\Sigma_{3,\Gamma}=&\text{fan}(\Sigma_{1,\Gamma_{e_1}},v_0)\cup\text{fan}(\Sigma_{1,\Gamma_{e_2}},v_0)\cup\text{fan}(\Sigma_{0,\Gamma_{e_3}},v_0),\\
\Sigma_{4,\Gamma}=&\text{fan}(\Sigma_{1,\Gamma_{e_1}},v_0)\cup\text{fan}(\Sigma_{1,\Gamma_{e_2}},v_0)\cup\text{fan}(\Sigma_{1,\Gamma_{e_3}},v_0).
\end{align*}
What happens now is that $\Sigma_{1,\Gamma_{e_j}}=\Sigma_{\Gamma_{e_j}}$ (this is the inductive step), and all that is left to see is that
\[
\Sigma_\Gamma=\text{fan}(\Sigma_{\Gamma_{e_1}},v_0)\cup\text{fan}(\Sigma_{\Gamma_{e_2}},v_0)\cup\text{fan}(\Sigma_{\Gamma_{e_3}},v_0).
\]
This follows from the equality
\[
\text{fan}(\Sigma_{\Gamma_{e_1}},v_0)=\Sigma_{\text{cone}(e_2,v_1,v_0)}\cup\Sigma_{\text{cone}(e_3,v_1,v_0)},
\]
which we will prove more generally in the sequel, and the similar equalities for $e_2$ and $e_3$.\par
  Let us start the proof of the Proposition. We begin by noting that the process described in the statement is well defined, i.e., we have that the cone $\tau_{j+1}$ belongs to the fan $\Sigma_{j,\Gamma}$ because $\dim \tau_{j+1}\leq\dim\tau_j$.\par
The proof is by induction on the number of edges. If $\Gamma$ has only one edge, or no edges, the result is trivial. Assume now that $\Gamma$ has at least two edges.\par
  If $\Gamma$ is not biconnected, then, by Lemma \ref{lem:prod}, $\Sigma_\Gamma=\prod \Sigma_{\Gamma_k}$ where $\Gamma_k$, for $k=1,\ldots,m$, are the biconnected components of $\Gamma$. Since every specialization $i_j$ factors through a specialization 
\[
\Gamma\leadsto \Gamma_k\leadsto \Gamma'_j,
\]
the result follows by the induction hypothesis applied to the graphs $\Gamma_i$ and by Remark \ref{rem:prodstar}.\par
  If $\Gamma$ is biconnected, then $\Gamma'_1=\Gamma$ and $\tau_1=\text{cone}(E(\Gamma))$. Let $v:=\sum_{e\in E(\Gamma)} e$. We have, by Definition \ref{def:star},
\[
\Sigma_{1,\Gamma}=\left\{\text{cone}(A)|A\subset E(\Gamma)\cup\left\{v\right\}\text{ and } E(\Gamma)\not\subset A\right\}.
\]
Let $\Gamma_e:=\Gamma/\{e\}$, for every $e\in E(\Gamma)$, and $i_e\colon\Gamma\leadsto \Gamma_e$ be the corresponding specialization (note that $i_e$ does not have to be in the good sequence). The fan $\Sigma_{1,\Gamma}$ can also be described as (recall Equation \eqref{eq:fanv}) 
\[
\Sigma_{1,\Gamma}=\bigcup_{e\in E(\Gamma)}\text{fan}(\Sigma_{0,\Gamma_e},v).
\]
Here, we regard the fans $\Sigma_{0,\Gamma}$ as subfans of $\Sigma_{1,\Gamma}$ via Proposition \ref{prop:face}. We note that the union in the right hand side is not disjoint, in fact the cone $\cone(E(\Gamma'))$ belongs to all fans $\Sigma_{0,\Gamma_e}$ such that there exists a specialization $\Gamma_e\leadsto \Gamma'$.  \par
The sequence $(i_j)$ induces a sequence $i_{j',e}\col \Gamma_e\leadsto \Gamma_{e,j'}$ of specializations of $\Gamma_e$, simply by omitting the specializations $i_j$ that do not factor through $\Gamma_e$. Moreover this sequence is a good sequence for $\Gamma_e$. Then, it follows by the induction hypothesis that $\Sigma_{\Gamma_e}$ can be obtained from $\Sigma_{0,\Gamma_e}$ via the same procedure described in the statement of the proposition. We also note that each nontrivial specialization of $\Gamma$ factors through a specialization of $\Gamma_e$ for some $e\in E(\Gamma)$. \par
  For the second step, we have the specialization $i_2\col\Gamma\leadsto\Gamma_2'$ with $\tau_2=\text{cone}(E(\Gamma_2'))$ and  $\Gamma'_2$ biconnected. By Remark \ref{rem:unionstar} and Proposition \ref{prop:fanv}, we have
\begin{align*}
\Sigma_{1,\Gamma}^\star(\tau_2)=&\left(\bigcup_{e\in E(\Gamma')}\text{fan}(\Sigma_{0,\Gamma_e},v)\right)\cup\left(\bigcup_{e\notin E(\Gamma')}\text{fan}(\Sigma_{0,\Gamma_{e}},v)^\star(\tau_2)\right)\\
=&\left(\bigcup_{e\in E(\Gamma')}\text{fan}(\Sigma_{0,\Gamma_e},v)\right)\cup\left(\bigcup_{e\notin E(\Gamma')}\text{fan}(\Sigma_{0,\Gamma_{e}}^\star(\tau_2),v)\right).
\end{align*}
In particular we can write
\[
\Sigma_{2,\Gamma}=\left(\bigcup_{e\in E(\Gamma')}\text{fan}(\Sigma_{0,\Gamma_e},v)\right)\cup\left(\bigcup_{e\notin E(\Gamma')}\text{fan}(\Sigma_{1,\Gamma_{e}},v)\right).
\]\par

Proceeding in the same fashion, we see that each new star subdivision of $\Sigma_{j,\Gamma}$ will be obtained by star subdivisions of $\Sigma_{i_{j,e},\Gamma_e}$, for some $i_{j,e}$. Since the good sequence $(i_j)$ for $\Gamma$ induces good sequences for each one of the $\Gamma_e$ we can write:
\[
\Sigma_{n,\Gamma}=\bigcup_{e\in E(\Gamma)}\text{fan}(\Sigma_{n_e,\Gamma_e},v).
\]
where $n_e$ is the number of terms in the good sequence for $\Gamma_e$. By the induction hypothesis, we have $\Sigma_{n_e,\Gamma_e}=\Sigma_{\Gamma_e}$, and hence
\[
\Sigma_{n,\Gamma}=\bigcup_{e\in E(\Gamma)}\text{fan}(\Sigma_{\Gamma_e},v).
\]
To end the proof, we only have to show that
\[
\bigcup_{e\in E(\Gamma)}\text{fan}(\Sigma_{\Gamma_e},v)=\Sigma_\Gamma.
\]
To prove that two fans are equal, it is enough to prove that they have the same maximal cones. The maximal cones in $\Sigma_\Gamma$ are of the form $\overline{K(\Gamma,p)}$ with p a generic enriched structure on $\Gamma$, while the maximal cones in $\text{fan}(\Sigma_{\Gamma_e},v)$ are of the form $\text{cone}(\overline{K(\Gamma_e,q)},v)$  with $q$ a generic enriched structure on $\Gamma_e$.\par
 By Condition $(2)$ in Definition \ref{def:enrichedgraph}, there exists a lower set $S\subset E(\Gamma)$ such that $e\<_p e'$ for every $e\in S$ and $e'\in E(\Gamma)$ and $(\Gamma/S,p|_{E(\Gamma/S)})$ is an enriched graph. Since $p$ is a generic enriched structure on $\Gamma$, $S$ has just one element, which we will call $e'$. We also have that $\Gamma/S=\Gamma_{e'}$. By Proposition \ref{prop:generators} and in special Equation \eqref{eq:iv}, we have that
\begin{equation}
\label{eq:generic}
\overline{K(\Gamma,p)}=\text{cone}\left(\overline{K(\Gamma_{e'},p|_{E(\Gamma_{e'})})}, v\right).
\end{equation}
Since $p$ is a generic, then so is $p|_{E(\Gamma_{e'})}$. This shows that every maximal cone in $\Sigma_\Gamma$ is a maximal cone in $\bigcup_{e\in E(\Gamma)}\text{fan}(\Sigma_{\Gamma_e},v)$. On the other hand, for every $q$ a generic enriched structure on $\Gamma_e$, there exists a generic enriched structure $p$ on $\Gamma$ such that $p|_{E(\Gamma_e)}=q$ (just define $e_1\<_pe_2$ if and only if $e_1=e$, or $e_1,e_2\neq e$ with $e_1\<_qe_2$). This shows, again by Equation \eqref{eq:generic}, that every maximal cone in $\bigcup_{e\in E(\Gamma)}\text{fan}(\Sigma_{\Gamma_e},v)$ is a maximal cone in $\Sigma_\Gamma$. This finishes the proof.
\end{proof}

\section{Enriched tropical curves and their moduli space}
\label{sec:moduli}

In this section we will construct the moduli space $E_g^{trop}$ parametrizing enriched tropical curves of genus $g$. Our construction follows the steps of the construction of $M_g^{trop}$ given in \cite{BMV}.\\

A \emph{vertex weighted enriched graph} is a 3-tuple $(\Gamma,p,w)$ such that $(\Gamma,p)$ is an enriched graph and $(\Gamma,w)$ is a vertex weighted graph. We say that $(\Gamma,p,w)$ \emph{specializes} to $(\Gamma',p',w')$ if $(\Gamma,p)$ (resp. $(\Gamma,w)$) specializes to $(\Gamma',p')$ (resp. $(\Gamma',w')$). Two vertex weighted enriched graphs $(\Gamma,p,w)$ and $(\Gamma',p',w')$ are \emph{isomorphic} if there exists an isomorphism $\rho$ of the vertex weighted graphs that preserves the enriched structure, i.e., $\rho(e)\<_{p'}\rho(e')$ if and only if $e\<_pe'$. We denote by $\Aut(\Gamma,p,w)$ the automorphism group of $(\Gamma,p,w)$. Clearly there is an injective group homomorphism 
\begin{equation}
\label{eq:aut}
\Aut(\Gamma,p,w)\hookrightarrow \Aut(\Gamma,w).
\end{equation}

An \emph{enriched tropical curve} is a 4-tuple $(\Gamma,p,w,l)$ where $(\Gamma,p,w)$ is a vertex weighted enriched graph and $l$ is a function $E(\Gamma)\to \mathbb{R}_{> 0}$ such that if $e\<_p e'$ then $l(e)\leq l(e')$ with equality if and only if $e\sim_p e'$. Clearly $(\Gamma,w,l)$ is a tropical curve. We define the genus of an enriched tropical curve as the genus of the underlying tropical curve. Two enriched tropical curves are \emph{isomorphic} if the underlying vertex weighted enriched graphs are isomorphic and the underlying tropical curves are isomorphic.\par

  Note that an isomorphism of the underlying tropical curves lifts to an isomorphism of the enriched tropical curves. Indeed we have that if $(\Gamma,p,w,l)$ and $(\Gamma,p',w,l)$ are enriched tropical curves, then $p=p'$ because $l$ induces a point in $\mathbb{R}_{>0}^{E(\Gamma)}$ that must belong to $K(\Gamma,p)$ and $K(\Gamma,p')$, which, by Proposition \ref{prop:union}, are disjoint if $p\neq p'$. Therefore, if $(\Gamma,p,w,l)$ and $(\Gamma',p',w',l')$ are enriched curves such that $(\Gamma,w,l)$ and $(\Gamma',w',l')$ are isomorphic, then, this isomorphism induces an enriched structure $p''$ on $(\Gamma,w,l)$ by pulling back $p'$ to $\Gamma$. Using the previous observation, we get that $p''=p$ and hence $(\Gamma,p,w,l)$ is isomorphic to $(\Gamma',p',w',l')$.\par
 Given a vertex weighted enriched graph $(\Gamma,p,w)$, we define the cone $K(\Gamma,p,w)$ as the cone $K(\Gamma,p)$. We also define the linear subspace $V_{\Gamma,p,w}\subset \mathbb{R}^{E(\Gamma)}$ as the minimal linear subspace containing $K(\Gamma,p,w)$. This subspace is given by equations $x_e=x_{e'}$ if $e\sim_p e'$, hence it is an integral linear subspace. Moreover the cone $K(\Gamma,p,w)$ has maximal dimension in $V_{\Gamma,p,w}$.\par 
Recall that $E(\Gamma)_p=E(\Gamma)/\sim_p$ and that $p$ induces a partial order in $E(\Gamma)_p$ such that its Hasse diagram is a rooted tree (see Proposition \ref{prop:tree} and Corollary \ref{cor:hasse}). Let $y_{[e]_p}$ be the coordinates of $\mathbb{R}^{E(\Gamma)_p}$. Define the map 
\[
\theta_{\Gamma,p,w}\colon V_{\Gamma,p,w}\to \mathbb{R}^{E(\Gamma)_p}
\]
as 
\[
\theta_{\Gamma,p,w}((x_e)_{e\in E(\Gamma)})=(y_{[e]_p})_{[e]_p\in E(\Gamma)_p}
\]
 with $y_{[e]_p}:=x_e-x_{e'}$, where $[e']_p$ and $[e]_p$ are consecutive with $e'\<_p e$ (in case $[e]_p$ is the root of the Hasse diagram, we define $y_{[e]_p}:=x_e$). Since the Hasse diagram of $E(\Gamma)_p$ is a rooted tree, the definition is well posed because, given a class $[e]_p$ different from the root, there exists a unique class $[e']_p$ such that $[e']_p$ and $[e]_p$ are consecutive with $e'\<_p e$. Moreover the definition does not depend on the choice of the representative in $[e]_p$ or $[e']_p$, because of the equations defining $V_{\Gamma,p,w}$. Clearly $\theta_{\Gamma,p,w}$ is linear and integral. It is an isomorphism because both spaces have the same dimension and it is injective. 

\begin{Lem}
\label{lem:theta}
We have that $\theta_{\Gamma,p,w}(K(\Gamma,p,w))=\mathbb{R}_{>0}^{E(\Gamma)_p}$. 
\end{Lem}
\begin{proof}
If $(x_e)\in K(\Gamma,p,w)$ then $x_e-x_{e'}\geq0$ for $e\>_pe'$ with equality if and only if $e\sim_pe'$. Therefore $y_{[e]_p}=x_e-x_{e'}>0$. On the other hand if $y_{[e]_p}>0$ for every equivalence class $[e]_p$, then $x_e>x_{e'}$ for $[e']_p$ and $[e]_p$ consecutive with $e'\<_p e$. Hence $x_e>x_{e'}$ for every $e>_pe'$ and $(y_{[e]_p})=\theta_{\Gamma,p,w}(x_e)$ with $(x_e)\in K(\Gamma,p,w)$.
\end{proof}
From now on we will identify $V_{\Gamma,p,w}$ and $K(\Gamma,p,w)$ with $\mathbb{R}^{E(\Gamma)_p}$ and $\mathbb{R}^{E(\Gamma)_p}_{>0}$, respectively, via $\theta_{\Gamma,p,w}$.\\

The automorphism group $\Aut(\Gamma,p,w)$ comes with two natural homomorphisms
\[
\Aut(\Gamma,p,w)\lra S_{|E(\Gamma)|} \subset GL_{|E(\Gamma)|}(\mathbb{Z})
\]
and
\[
\Aut(\Gamma,p,w)\lra S_{|E(\Gamma)_p|} \subset GL_{|E(\Gamma)_p|}(\mathbb{Z}).
\]
The latter homomorphism comes from the fact that if $e\sim_pe'$ then $\rho(e)\sim_p\rho(e')$ for each $\rho\in\Aut(\Gamma,p,w)$.\par
  The first homomorphism induces an action of $\Aut(\Gamma,p,w)$ on $\mathbb{R}^{E(\Gamma)}$ that keeps invariant the subspace $V_{\Gamma,p,w}=\mathbb{R}^{E(\Gamma)_p}$. Moreover the restriction of this action to $V_{\Gamma,p,w}=\mathbb{R}^{E(\Gamma)_p}$ is the action of $\Aut(\Gamma,p,w)$ induced by the second homomorphism. Both actions preserve the two cones $K(\Gamma,p,w)$ and $\overline{K(\Gamma,p,w)}$.\par
	Moreover given an element of $\Aut(\Gamma,w)$, its action on $\mathbb{R}^{E(\Gamma)}$ induces an action on $\Sigma_\Gamma$.\par
 We then define the quotients
\[
C(\Gamma,p,w):=\frac{K(\Gamma,p,w)}{\Aut(\Gamma,p,w)}=\frac{\mathbb{R}^{E(\Gamma)_p}_{>0}}{\Aut(\Gamma,p,w)}
\]
and
\[
\overline{C(\Gamma,p,w)}:=\frac{\overline{K(\Gamma,p,w)}}{\Aut(\Gamma,p,w)}=\frac{\mathbb{R}^{E(\Gamma)_p}_{\geq0}}{\Aut(\Gamma,p,w)}
\]
endowed with the quotient topology. When $\Gamma$ has no edges we define $C(\Gamma,p,w):=\{0\}$. Note that $C(\Gamma,p,w)$ parametrizes isomorphism classes of enriched tropical curves with underlying vertex weighted enriched graph $(\Gamma,p,w)$.\par

  By Proposition \ref{prop:face}, if $i\colon(\Gamma,p,w)\leadsto (\Gamma',p',w')$ is a specialization then the linear map 
\[
\overline{i}\colon \mathbb{R}^{E(\Gamma')}\hookrightarrow \mathbb{R}^{E(\Gamma)}
\]
induces a map
\begin{equation}
\label{eq:iK}
\overline{i}\colon \overline{K(\Gamma',p',w')}\hookrightarrow \overline{K(\Gamma,p,w)}\lra \overline{C(\Gamma,p,w)}. 
\end{equation}

  We can also define the natural projection map 
\begin{equation}
\label{eq:gammap}
f\colon \mathbb{R}^{E(\Gamma)_p}\lra \mathbb{R}^{E(\Gamma')_{p'}}
\end{equation}
given by the natural inclusion $E(\Gamma')_{p'}\hookrightarrow E(\Gamma)_p$ as in Corollary \ref{cor:inclusion}.

\begin{Def}
We define $E_g^{trop}$ as the topological space (with respect to the quotient topology)
\[
E_g^{trop}:=\left(\coprod \overline{C(\Gamma,p,w)}\right)_{/\sim}
\]
where the disjoint union runs through all (isomorphism classes of) stable vertex weighted enriched graphs $(\Gamma,p,w)$ of genus $g$ and $\sim$ is the equivalence relation given as follows: $q_1\in \overline{C(\Gamma_1,p_1,w_1)}$ and $q_2\in \overline{C(\Gamma_2,p_2,w_2)}$ are equivalent if and only if there exist a point $q\in \overline{K(\Gamma,p,w)}$ and specializations $i_1\colon (\Gamma_1,p_1,w_1)\leadsto (\Gamma,p,w)$ and $i_2\colon(\Gamma_2,p_2,w_2)\leadsto(\Gamma,p,w)$ such that $\overline{i_1}(q)=q_1$ and $\overline{i_2}(q)=q_2$. It is not hard to show that $\sim$ is indeed transitive.
\end{Def}

\begin{Exa}
In Figure \ref{fig:g2} we show all the tropical enriched curves of genus 2 and their specializations. In genus 2 we have exactly one vertex weighted biconnected graph, which we will call $(\Gamma,w)$, and which gives rise to 3 enriched structures, the generic one of which we denote by $p$. On the remaining vertex weighted graphs, the enriched structures are trivial.

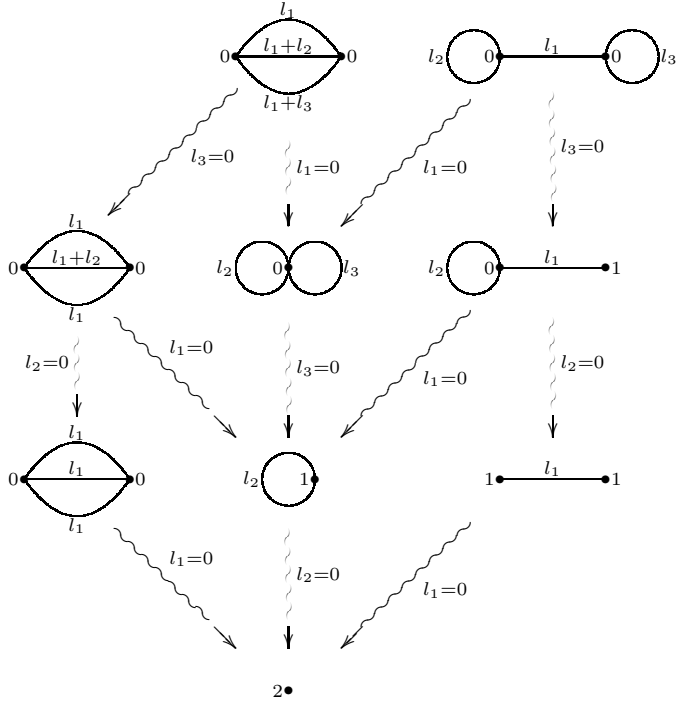
\begin{figure}[h]
\[
\begin{xy} <40pt,0pt>:
(0,0)*{\scriptstyle\bullet}="a"; 
(1,0)*{\scriptstyle\bullet}="b";
(2.5,0)*{\scriptstyle\bullet}="c";
(3.5,0)*{\scriptstyle\bullet}="d";
(2.5,-2)*{\scriptstyle\bullet}="e";
(3.5,-2)*{\scriptstyle\bullet}="f";
(0.5,-2)*{\scriptstyle\bullet}="g";
(-2,-2)*{\scriptstyle\bullet}="h"; 
(-1,-2)*{\scriptstyle\bullet}="i";
(-2,-4)*{\scriptstyle\bullet}="j"; 
(-1,-4)*{\scriptstyle\bullet}="l";
(0.75,-4)*{\scriptstyle\bullet}="m";
(2.5,-4)*{\scriptstyle\bullet}="n";
(3.5,-4)*{\scriptstyle\bullet}="o";
(0.5,-6)*{\scriptstyle\bullet}="p";
"a"+0;"b"+0**\crv{"a"+(0.5,0.7)};
"a"+0;"b"+0**\crv{"a"+(0.5,-0.7)};
"a"+0;"b"+0**\crv{"a"+(0.5,0)};
"c"+0;"d"+0**\crv{"c"+(0.5,0)}; 
"c"+(-0.25,0)*\xycircle(0.25,0.25){-};
"d"+(0.25,0)*\xycircle(0.25,0.25){-};
"e"+0;"f"+0**\crv{"e"+(0.5,0)}; 
"e"+(-0.25,0)*\xycircle(0.25,0.25){-};
"g"+(-0.25,0)*\xycircle(0.25,0.25){-};
"g"+(0.25,0)*\xycircle(0.25,0.25){-};
"h"+0;"i"+0**\crv{"h"+(0.5,0.7)};
"h"+0;"i"+0**\crv{"h"+(0.5,-0.7)};
"h"+0;"i"+0**\crv{"h"+(0.5,0)};
"j"+0;"l"+0**\crv{"j"+(0.5,0.7)};
"j"+0;"l"+0**\crv{"j"+(0.5,-0.7)};
"j"+0;"l"+0**\crv{"j"+(0.5,0)};
"m"+(-0.25,0)*\xycircle(0.25,0.25){-};
"n"+0;"o"+0**\crv{"n"+(0.5,0)}; 
"a"+(0.5,0.44)*{\scriptstyle{l_1}};
"a"+(0.5,0.1)*{\scriptstyle{l_1+l_2}};
"a"+(0.5,-0.44)*{\scriptstyle{l_1+l_3}};
"h"+(0.5,0.44)*{\scriptstyle{l_1}};
"h"+(0.5,0.1)*{\scriptstyle{l_1+l_2}};
"h"+(0.5,-0.44)*{\scriptstyle{l_1}};
"j"+(0.5,0.44)*{\scriptstyle{l_1}};
"j"+(0.5,0.1)*{\scriptstyle{l_1}};
"j"+(0.5,-0.44)*{\scriptstyle{l_1}};
"c"+(-0.6,0)*{\scriptstyle{l_2}};
"c"+(0.5,0.1)*{\scriptstyle{l_1}};
"d"+(0.6,0)*{\scriptstyle{l_3}};
"e"+(-0.6,0)*{\scriptstyle{l_2}};
"e"+(0.5,0.1)*{\scriptstyle{l_1}};
"n"+(0.5,0.1)*{\scriptstyle{l_1}};
"g"+(-0.6,0)*{\scriptstyle{l_2}};
"g"+(0.6,0)*{\scriptstyle{l_3}};
"m"+(-0.6,0)*{\scriptstyle{l_2}};
"a"+(-0.1,0)*{\scriptstyle{0}};
"b"+(0.1,0)*{\scriptstyle{0}};
"c"+(-0.1,0)*{\scriptstyle{0}};
"d"+(0.1,0)*{\scriptstyle{0}};
"e"+(-0.1,0)*{\scriptstyle{0}};
"f"+(0.1,0)*{\scriptstyle{1}};
"g"+(-0.1,0)*{\scriptstyle{0}};
"h"+(-0.1,0)*{\scriptstyle{0}};
"i"+(0.1,0)*{\scriptstyle{0}};
"j"+(-0.1,0)*{\scriptstyle{0}};
"l"+(0.1,0)*{\scriptstyle{0}};
"m"+(-0.1,0)*{\scriptstyle{1}};
"n"+(-0.1,0)*{\scriptstyle{1}};
"o"+(0.1,0)*{\scriptstyle{1}};
"p"+(-0.1,0)*{\scriptstyle{2}};
{\ar^{\scriptstyle{l_3=0}}@{~}(0,-0.3)*{};(-1,-1.3)*{}};
{\ar@{->}(-1,-1.3)*{};(-1.2,-1.5)*{}};
{\ar^{\scriptstyle{l_1=0}}@{~}(0.5,-0.7)*{};(0.5,-1.4)*{}};
{\ar@{->}(0.5,-1.4)*{};(0.5,-1.6)*{}};
{\ar^{\scriptstyle{l_1=0}}@{~}(2.2,-0.4)*{};(1.2,-1.4)*{}};
{\ar@{->}(1.2,-1.4)*{};(1,-1.6)*{}};
{\ar^{\scriptstyle{l_3=0}}@{~}(3,-0.3)*{};(3,-1.4)*{}};
{\ar@{->}(3,-1.4)*{};(3,-1.6)*{}};
{\ar^{\scriptstyle{l_3=0}}@{~}(0.5,-2.5)*{};(0.5,-3.4)*{}};
{\ar@{->}(0.5,-3.4)*{};(0.5,-3.6)*{}};
{\ar_{\scriptstyle{l_2=0}}@{~}(-1.5,-2.6)*{};(-1.5,-3.2)*{}};
{\ar@{->}(-1.5,-3.2)*{};(-1.5,-3.4)*{}};
{\ar^{\scriptstyle{l_2=0}}@{~}(3,-2.4)*{};(3,-3.4)*{}};
{\ar@{->}(3,-3.4)*{};(3,-3.6)*{}};
{\ar^{\scriptstyle{l_1=0}}@{~}(2.2,-2.4)*{};(1.2,-3.4)*{}};
{\ar@{->}(1.2,-3.4)*{};(1,-3.6)*{}};
{\ar^{\scriptstyle{l_1=0}}@{~}(-1.2,-2.4)*{};(-0.2,-3.4)*{}};
{\ar@{->}(-0.2,-3.4)*{};(0,-3.6)*{}};
{\ar^{\scriptstyle{l_2=0}}@{~}(0.5,-4.4)*{};(0.5,-5.4)*{}};
{\ar@{->}(0.5,-5.4)*{};(0.5,-5.6)*{}};
{\ar^{\scriptstyle{l_1=0}}@{~}(2.2,-4.4)*{};(1.2,-5.4)*{}};
{\ar@{->}(1.2,-5.4)*{};(1,-5.6)*{}};
{\ar^{\scriptstyle{l_1=0}}@{~}(-1.2,-4.4)*{};(-0.2,-5.4)*{}};
{\ar@{->}(-0.2,-5.4)*{};(0,-5.6)*{}};
\end{xy}
\]
\caption{Genus-$2$ enriched tropical curves.}
\label{fig:g2}
\end{figure}

As in \cite[Example 3.2.3]{BMV} one can find the maximal cells of $E_2^{trop}$. In fact, using Lemma \ref{lem:theta}, these cells are $\overline{C(\Gamma,p,w)}=\mathbb{R}^{3}_{\geq0}/\Aut(\Gamma,p,w)$ and $\mathbb{R}^3_{\geq0}/S_2$, corresponding to the vertex weighted graphs on the top of Figure \ref{fig:g2}. Moreover, as Theorem \ref{thm:maintrop} states, there is map $E_2^{trop}\to M_2^{trop}$ which restricted to the cell $\overline{C(\Gamma,p,w)}$ is just:
\begin{eqnarray*}
\frac{\mathbb{R}^3_{\geq0}}{\Aut(\Gamma,p,w)}&\lra&\frac{\mathbb{R}^3_{\geq0}}{\Aut(\Gamma,w)}\\
(l_1,l_2,l_3)&\mapsto  &(l_1,l_1+l_2,l_1+l_3),
\end{eqnarray*}
where $\Aut(\Gamma,p,w)=S_2$ acts on $\mathbb{R}^3_{\geq0}$ by permuting the last two coordinates and $\Aut(\Gamma,w)=S_3$ acts naturally on $\mathbb{R}^3_{\geq0}$.
\end{Exa}

\begin{Thm}
\label{thm:maintrop}
The topological space $E_g^{trop}$ is a stacky fan with cells $C(\Gamma,p,w)$, as $(\Gamma,p,w)$ varies through all (isomorphism classes of) stable vertex weighted enriched graphs of genus $g$. In particular, its points are in bijection with the isomorphism classes of stable enriched tropical curves of genus $g$. Moreover there exists a bijective map of stacky fans
\[
\beta\colon E_g^{trop}\to M_g^{trop}
\]
of degree one.

\end{Thm}
\begin{proof}
  First we prove that $E_g^{trop}$ is a stacky fan with cells $C(\Gamma,p,w)$. Consider the maps 
\[
\alpha_{\Gamma,p,w}\colon\overline{C(\Gamma,p,w)}\lra \coprod \overline{C(\Gamma,p,w)}\lra E_g^{trop}
\]
defined as the composition of the quotient map and the natural inclusion. Clearly $\alpha_{\Gamma,p,w}$ is continuous. Moreover, by Propositions \ref{prop:union} and \ref{prop:unionK}, the restriction of $\alpha_{\Gamma,p,w}$ to $C(\Gamma,p,w)$ is an homeomorphism onto its image and we have 
\[
E_g^{trop}=\coprod\alpha_{\Gamma,p,w}(C(\Gamma,p,w)).
\]

 To prove the last condition in Definition \ref{def:stackyfan}, let $(\Gamma,p,w)$ and $(\Gamma',p',w')$ be two stable vertex weighted enriched graphs and set $\alpha:=\alpha_{\Gamma,p,w}$ and $\alpha':=\alpha_{\Gamma',p',w'}$. By construction the intersection of the images of $\overline{C(\Gamma,p,w)}$ and $\overline{C(\Gamma',p',w')}$ in $E_g^{trop}$ is 
\[
\alpha(C(\Gamma,p,w))\cap \alpha'(C(\Gamma',p',w'))=\coprod \alpha_i(C(\Gamma_i,p_i,w_i))
\]
where $(\Gamma_i,p_i,w_i)$ runs through all common specializations of both $(\Gamma,p,w)$ and $(\Gamma',p',w')$ and $\alpha_i:=\alpha_{\Gamma_i,p_i,w_i}$. \par
We have to find an integral linear map $L\colon V_{\Gamma,p,w}\to V_{\Gamma',p',w'}$ such that the restriction of $L$ to $\overline{K(\Gamma,p,w)}$ makes the diagram below commutative.

\[
\SelectTips{cm}{11}
\begin{xy}<14pt,0pt>:
\xymatrix@C-=0.5cm{\coprod\alpha_i(C(\Gamma_i,p_i,w_i))\UseTips\ar@{^{(}->}[r] \ar@{^{(}->}[dr] &\alpha(\overline{C(\Gamma,p,w)}) & \ar[l] \overline{K(\Gamma,p,w)} \ar@{^{(}->}[r]\ar[d] & V_{\Gamma,p,w}\ar[d]^{L}\ar[r]^{\sim} & \mathbb{R}^{E(\Gamma)_p}\ar[d]^{L} \\
 &\alpha'(\overline{C(\Gamma',p',w')}) & \ar[l] \overline{K(\Gamma',p',w')} \ar@{^{(}->}[r] & V_{\Gamma',p',w'}\ar[r]^{\sim}&\mathbb{R}^{E(\Gamma')_{p'}}}
\end{xy}
\]

 To construct $L$, we note that if $(\Gamma_i,p_i,w_i)$ is a common specialization of both $(\Gamma,p,w)$ and $(\Gamma',p',w')$, then there is a linear integral map $f_i\colon\mathbb{R}^{E(\Gamma)_p}\to\mathbb{R}^{E(\Gamma_i)_{p_i}}$, as in Equation \eqref{eq:gammap}, and an inclusion $g_i\colon V_{\Gamma_i,p_i,w_i}\hookrightarrow V_{\Gamma',p',w'}$ induced by the inclusion $\mathbb{R}^{E(\Gamma_i)}\hookrightarrow \mathbb{R}^{E(\Gamma')}$. Then define $L$ as the composition 
\[
L\colon V_{\Gamma,p,w}\stackrel{\oplus f_i}{\lra} \oplus_iV_{\Gamma_i,p_i,w_i}\stackrel{\oplus g_i}{\lra}V_{\Gamma',p',w'}.
\]
Since  $L$ is an integral linear map, we need only prove that 
\[
L(\overline{K(\Gamma,p,w)})\subset\overline{K(\Gamma',p',w')}.
\]
 Noting that $f_i(\mathbb{R}^{E(\Gamma)_p}_{>0})=\mathbb{R}^{E(\Gamma_i)_{p_i}}_{>0}$ and recalling that there is an identification between $K(\Gamma,p,w)$ and $\mathbb{R}^{E(\Gamma)_p}_{>0}$, we have
\[
f_i(\overline{K(\Gamma,p,w)})\subset \overline{K(\Gamma_i,p_i,w_i)}.
\]
Moreover, by Proposition \ref{prop:face}, we have
\[
g_i(\overline{K(\Gamma_i,p_i,w_i)}\subset\overline{K(\Gamma',p',w')}.
\]
Since $\overline{K(\Gamma',p',w')}$ is a cone, the image of $\oplus \overline{K(\Gamma_i,p_i,w_i)}$ via $\oplus g_i$ is also contained in $\overline{K(\Gamma',p',w')}$.
  Clearly $L$ makes the above diagram commutative hence the proof of the first statement is complete.\par
  The second statement follows from the fact that 
	\[
	E_g^{trop}=\coprod \alpha_{\Gamma,p,w}C(\Gamma,p,w)
	\]
and the fact that $C(\Gamma,p,w)$ parametrizes isomorphism classes of stable tropical enriched curves with underlying vertex weighted enriched graph $(\Gamma,p,w)$.\par
   Let us prove the last sentence. First we recall that there is an inclusion map $\overline{K(\Gamma,p,w)}\hookrightarrow \mathbb{R}_{\geq0}^{E(\Gamma)}$, and by the injection in Equation \eqref{eq:aut}, we have a continuous map 
\[
\beta_{\Gamma,p,w}\colon\overline{C(\Gamma,p,w)}\to\overline{C(\Gamma,w)}
\]
making the following diagram commutative
\begin{equation}
\label{eq:diagrambeta}
\SelectTips{cm}{11}
\begin{xy}<16pt,0pt>:
\xymatrix@+=1cm{\overline{K(\Gamma,p,w)}\ar@{^{(}->}[r]\ar[d] & \mathbb{R}_{\geq0}^{E(\Gamma)}\ar[d]\\
\overline{C(\Gamma,p,w)}\ar[r]^{\beta_{\Gamma,p,w}}  & \overline{C(\Gamma,w)}}
\end{xy}
\end{equation}
where the vertical maps are the natural quotient maps.\par
The maps $\beta_{\Gamma,p,w}$ induce a continuous map 
\begin{equation}
\label{eq:coprodmap}
\coprod\overline{C(\Gamma,p,w)}\lra\coprod \overline{C(\Gamma,w)}.
\end{equation}
Hence to construct the natural forgetful map 
\[
\beta\colon E_g^{trop}\to M_g^{trop}
\]
it is enough to check that if two points $q_1\in\overline{C(\Gamma_1,p_1,w_1)}$ and $q_2\in\overline{C(\Gamma_2,p_2,w_2)}$ are equivalent, then the images $\beta_{\Gamma_1,p_1,w_1}(q_1)$ and $\beta_{\Gamma_2,p_2,w_2}(q_2)$ are equivalent. Indeed consider specializations
\[
i_1\colon (\Gamma_1,p_1,w_1)\leadsto (\Gamma,p,w) \quad \text{and} \quad i_2\colon(\Gamma_2,p_2,w_2)\leadsto (\Gamma,p,w)
\]
 with underlying specializations 
\[
j_1\colon (\Gamma_1,w_1)\leadsto (\Gamma,w) \quad \text{and} \quad j_2\colon(\Gamma_2,w_2)\leadsto(\Gamma,w).
\]
Denote by 
\[
\overline{i_k}\colon \overline{K(\Gamma,p,w)}\to\overline{C(\Gamma_k,p_k,w_k)}
\]
and by 
\[
\overline{j_k}\colon\mathbb{R}_{\geq0}^{E(\Gamma)}\to\overline{C(\Gamma,w)}
\]
the induced maps as in Equation \eqref{eq:iK} and in \cite[Section 3.2]{BMV}. All that is left to note is that if $q\in\overline{K(\Gamma,p,w)}$ is such that $\overline{i_1}(q)=q_1$ and $\overline{i_2}(q)=q_2$ then $\overline{j_k}(q)=\beta_{\Gamma_k,p_k,w_k}(q_k)$ for $k=1,2$. This comes from the following commutative diagram obtained by Diagram \eqref{eq:diagrambeta} and Proposition \ref{prop:face}
\[
\SelectTips{cm}{11}
\begin{xy}<16pt,0pt>:
\xymatrix@+=1cm{\overline{K(\Gamma,p,w)}\ar@{^{(}->}[r]\ar@{^{(}->}[d] & \overline{K(\Gamma_1,p_1,w_1)}\ar@{^{(}->}[d] \ar[r] & \overline{C(\Gamma_1,p_1,w_1)} \ar[d]^{\beta_{\Gamma_1,p_1,w_1}}\\
\mathbb{R}_{\geq0}^{E(\Gamma)}\ar@{^{(}->}[r]& \mathbb{R}_{\geq0}^{E(\Gamma_1)}\ar[r] & \overline{C(\Gamma_1,w)}}
\end{xy}
\]

The map $\beta$ is continuous because so is the map in Equation \eqref{eq:coprodmap}. Moreover $\beta$ is a map of stacky fans because, by Diagram \eqref{eq:diagrambeta}, we have $\beta_{\Gamma,p,w}(C(\Gamma,p,w))\subset C(\Gamma,w)$ and there is a commutative diagram 
\[
\SelectTips{cm}{11}
\begin{xy}<16pt,0pt>:
\xymatrix@+=1cm{C(\Gamma,p,w)\ar[d]^{\beta_{\Gamma,p,w}} &\ar[l] K(\Gamma,p,w)\ar@{^{(}->}[r]\ar@{^{(}->}[d]  & V_{\Gamma,p,w}\ar@{^{(}->}[d]\\
C(\Gamma,w)& \ar[l]\mathbb{R}_{>0}^{E(\Gamma)}\ar[r] & \mathbb{R}^{E(\Gamma)}}
\end{xy}
\]

  To prove that $\beta$ is bijective we first prove that it is surjective. Let $(\Gamma,w)$ be a vertex weighted graph and $\overline{q}\in C(\Gamma,w)$, and let $q\in\mathbb{R}^{E(\Gamma)}_{>0}$ be a representative of $\overline{q}$. By Proposition \ref{prop:union}, there exists an enriched structure $p$ such that $q\in K(\Gamma,p,w)$. It follows that, if $\widetilde{q}\in C(\Gamma,p,w)$ is the class of $q$, then $\beta(\widetilde{q})=\overline{q}$.\par
	We now prove that $\beta$ is injective.  Indeed if $\widetilde{q_1}\in C(\Gamma,p_1,w)$ and $\widetilde{q_2}\in C(\Gamma,p_2,w)$, where $q_1,q_2\in \mathbb{R}^{E(\Gamma)}_{>0}$ are such that $\beta(\widetilde{q_1})=\beta(\widetilde{q_2})$ then there exists an automorphism $\sigma$ of $(\Gamma,w)$ such that $\sigma(q_1)=q_2$. Then $q_2$ belongs to the cone $K(\Gamma,p_1^{\sigma},w)$, where $p_1^{\sigma}$ is the enriched structure induced by $\sigma$. By Proposition \ref{prop:union} $p_1^{\sigma}=p_2$, and hence $\sigma$ is an isomorphism between $(\Gamma,p_1,w)$ and $(\Gamma,p_2,w)$ and we deduce that $\widetilde{q_1}=\widetilde{q_2}$.\par
	To check that $\beta$ is of degree one all that is left is to see is that the map $V_{\Gamma,p,w}\to\mathbb{R}^{E(\Gamma)}$ is primitive, which is true because it is an inclusion and the lattice of $V_{\Gamma,p,w}$ is the lattice of $\mathbb{R}^{E(\Gamma)}$ intersected with $V_{\Gamma,p,w}$.
\end{proof}

\begin{Rem}
The map $\beta$ is not full for $g\geq 2$. Indeed consider the cone $C(\Gamma,p,w)$ where $\Gamma$ is biconnected and $p$ is such that $e\sim_p e'$ for all $e,e'\in E(\Gamma)$. Then $C(\Gamma,p,w)$ is one dimensional, because $\rank(p)=1$, and its image via $\beta$ lies inside $C(\Gamma,w)$ which has dimension $|E(\Gamma)|$, hence such image is properly contained in $C(\Gamma,w)$ if $\Gamma$ has at least $2$ edges.
\end{Rem}

In \cite[Proposition 3.3.3]{Caporaso1}, Caporaso proves that $M_g^{trop}$ is connected through codimension one. Moreover a description of the maximal and codimension one cells of $M_g^{trop}$ is given in \cite[Proposition 3.2.5]{BMV}. We now prove similar results for $E_g^{trop}$.

\begin{Prop}
\label{prop:cod1}
The following properties hold.
\begin{enumerate}[label=(\roman*)]
\item The maximal cells of $E_g^{trop}$ are exactly those of the form $C(\Gamma,p,\underline{0})$ where $\Gamma$ is a $3$-regular graph and $(\Gamma,p)$ is a generic enriched graph. In particular $E_g^{trop}$ has pure dimension $3g-3$.
\item $E_g^{trop}$ is connected through codimension one.
\item The codimension one cells of $E_g^{trop}$ are of the following types
  \begin{enumerate}[label=(\alph*)]
  \item $C(\Gamma,p,\underline{0})$ where $\Gamma$ has exactly one vertex of valence $4$ and all other vertices of valence $3$ and $(\Gamma,p)$ is generic;
	\item $C(\Gamma,p,w)$ where $\Gamma$ has exactly one vertex of valence $1$ and weight $1$, and all other vertices of valence $3$ and weight $0$, and $(\Gamma,p)$ is generic.
	\item $(\Gamma,p,\underline{0})$ where $\Gamma$ is a $3$-regular graph and $(\Gamma,p)$ is a simple specialization of a generic enriched graph.
  \end{enumerate}
	Each cell of type (a) is in the closure of one, two or three maximal cells; each cell of type (b) is in the closure of exactly one maximal cell and each cell of type (c) is in the closure of one or two maximal cells.
\end{enumerate}
\end{Prop}
\begin{proof}
Statement (i) follows from the following facts. First, the number of edges of $\Gamma$ is limited by $3g-3$ and equality occurs if $\Gamma$ is $3$-regular (by \cite[proof of Proposition 3.2.5]{BMV}). Second, the dimension of $C(\Gamma,p,w)$ is the rank of $p$ (by Equation \eqref{eq:dimrank}) that is maximum if and only if $p$ is generic. Moreover every vertex weighted enriched graph is the specialization of $(\Gamma,p,\underline{0})$ for some $3$-regular graph $\Gamma$ and $p$ generic, see Corollary \ref{cor:generic} and \cite[Appendix A.2]{CV}, which concludes the proof of Statement (i).\par
 Statement (ii) follows by \cite[Proposition 3.3.3]{Caporaso1} and \cite[Statement (ii) in Proposition 3.2.5]{BMV} and from the fact that the fan $\Sigma_\Gamma$, defined in Definition \ref{def:fan}, is connected through codimension one for every $\Gamma$, because its support is pure dimensional.  \par
 We prove Statement (iii). Let $C(\Gamma,p,w)$ be a codimension one cell.  Note that since $\beta$ is of degree one, the image $\beta(C(\Gamma,p,w))\subset C(\Gamma,w)$ is either contained in a maximal cell or in a codimension one cell of $M_g^{trop}$. Hence, we have that $(\Gamma,w)$ is in one of the cases of \cite[Statements (i) and (iii) in Proposition 3.2.5]{BMV}. In the first case $\rank(p)$ must be $3g-2$ and then, by Proposition \ref{prop:simple} and Corollary \ref{cor:generic}, $(\Gamma,p)$ is a simple specialization of a generic enriched graph. In the second case $p$ must be generic.\par
  The last statement for cells of type (a) and (b) follows from the similar result in \cite[Proposition 3.2.5]{BMV}. To prove the statement for a cell of type (c), it is enough to prove that given a vertex weighted enriched graph $(\Gamma,p,w)$ such that $\rank(p)=|E(\Gamma)|-1$, then there exist exactly $2$ generic enriched structures $p_1$ and $p_2$ on $\Gamma$ and simple specializations $(\Gamma,p_1,w)\leadsto (\Gamma,p,w)$ and $(\Gamma,p_2,w)\leadsto (\Gamma,p,w)$. This follows from Proposition \ref{prop:union}. We note, however, that these two vertex weighted enriched graphs could be isomorphic, and in this case a cell of type (c) is in the closure of only one maximal cell.\end{proof}

\section{The variety of enriched structures}
\label{sec:enriched}

In this section we return to algebraic geometry. We construct the toric variety $\E_Y$ associated to a nodal curve $Y$ and, in Corollary \ref{cor:maino}, we prove that $\E_Y$ is the fiber of the map $\overline{\E}_g  \to\overline{\M}_g$ over the point $[Y]\in\overline{\M}_g$, in the case $Y$ has no nontrivial automorphisms.\par 
Let $Y$ be a nodal curve, $\Gamma$ be its dual graph and $\Sigma_\Gamma$ be the fan defined in Definition \ref{def:fan}.  We define the variety $\E_Y$ as the preimage of the fixed point (i.e., the origin) in $X\left(\mathbb{R}^{E(\Gamma)}_{\geq0}\right)\simeq \mathbb{A}^{|E(\Gamma)|}$ via the map
\[
X(\Sigma_\Gamma)\to X\left(\mathbb{R}^{E(\Gamma)}_{\geq0}\right).
\]
We call $\E_Y$ the \emph{variety of enriched structures} on $Y$.\par

 Note that $\E_Y$ is the toric variety associated to the fan $\Sigma_Y$ defined as the image of $\Sigma_\Gamma$ via the quotient map
\begin{equation}
\label{eq:Vgamma}
q_\Gamma\col\mathbb{R}^{E(\Gamma)}\lra\frac{\mathbb{R}^{E(\Gamma)}}{U_\Gamma},
\end{equation}
where $U_\Gamma$ is the integral linear subspace generated by the vectors $v_i:=\sum_{e\in \Gamma_i} e$ and $\Gamma_i$, $i=1,\ldots, m$, are the biconnected components of $\Gamma$. Indeed, the preimage of the fixed point in $\mathbb{A}^{E(\Gamma)}$ is, by \cite[Lemma 3.3.21]{CLS}, exactly
\[
\underset{\sigma\cap\mathbb{R}_{>0}^{E(\Gamma)}\neq\emptyset}{\bigcup_{\sigma\in\Sigma_\Gamma}}V(\sigma).
\]
Every cone $\sigma\in\Sigma_{\Gamma}$ with $\sigma\cap\mathbb{R}_{>0}^{E(\Gamma)}\neq\emptyset$ is of the form $\overline{K(\Gamma,p)}$ for some enriched structure $p$ on $\Gamma$. By Proposition \ref{prop:generators}, we have that, for every $i=1,\ldots,m$, the vector $v_i$ is a ray generator of all these cones (just take $T=E(\Gamma_i)$). Hence, we conclude that every cone $\sigma\in\Sigma_\Gamma$ with $\sigma\cap\mathbb{R}^{E(\Gamma)}_{>0}\neq\emptyset$ has $\text{cone}(v_1,\ldots,v_m)$ as a face (because, by Proposition \ref{prop:smooth}, $\sigma$ is smooth). This implies that $V(\sigma)\subset V(\text{cone}(v_1,\ldots,v_m))$. Therefore $\E_Y=V(\text{cone}(v_1,\ldots,v_m))$, and by Remark \ref{rem:subvariety}, we have the above description for the fan $\Sigma_Y$.

\begin{Exa} 
Let $Y$ be a  curve with $2$ components that intersect each other in $3$ nodes. Then its dual graph $\Gamma$ is a graph with $2$ vertices and $3$ edges $e_1,e_2,e_3$ joining these two vertices. Let $v:=e_1+e_2+e_3\in\mathbb{R}^{E(\Gamma)}$. The fan $\Sigma_{\Gamma}$ is described in Example \ref{ex:fans}. In Figure \ref{fig:2c'} we depict a section of this fan. The fan $\Sigma_Y$ is the quotient of $\Sigma_\Gamma$ by $v$ and is described in Figure \ref{fig:p2},  where $\overline{e_i}$ is the class of $e_i$ modulo $v$. We note that $\overline{e_1}+\overline{e_2}+\overline{e_3}=0$, which implies that $\Sigma_Y$  is the fan of $\mathbb{P}^2$.
\begin{figure}[htb]
\begin{minipage}{0.48\linewidth}
\[
\begin{xy} <10pt,0pt>:
(-4,0)*{\scriptstyle\bullet}="a"; 
(4,0)*{\scriptstyle\bullet}="b";
(0,6)*{\scriptstyle\bullet}="c";
(0,2)*{\scriptstyle\bullet}="v";
{\ar@{-}"a"*{};"b"*{}};
{\ar@{-}"a"*{};"c"*{}};
{\ar@{-}"a"*{};"v"*{}};
{\ar@{-}"c"*{};"b"*{}};
{\ar@{-}"v"*{};"b"*{}};
{\ar@{-}"v"*{};"c"*{}};
"v"+(-0.23,-0.55)*{v};
"a"+(-0.45,-0.25)*{e_2};
"b"+(0.6,-0.25)*{e_1};
"c"+(0.25,0.35)*{e_3};
\end{xy}
\]
\caption{A section of the fan $\Sigma_\Gamma$.}
\label{fig:2c'}
\end{minipage}
\quad
\begin{minipage}{0.48\linewidth}
\[
\begin{xy} <10pt,0pt>:
(-4,-4)="a"; 
(4,0)="b";
(0,4)="c";
(0,0)="v";
{\ar@{->}"v"*{};"a"*{}};
{\ar@{->}"v"*{};"b"*{}};
{\ar@{->}"v"*{};"c"*{}};
"a"+(-0.45,0.5)*{\overline{e_2}};
"b"+(0,-0.6)*{\overline{e_1}};
"c"+(0.25,0.45)*{\overline{e_3}};
\end{xy}
\]
\caption{The fan $\Sigma_Y$.}
\label{fig:p2}
\end{minipage}
\quad
\end{figure}
\end{Exa}

\begin{Exa}
Let $Y$ be a circular curve with $3$ components. Then its dual graph $\Gamma$ is a circular graph with $3$ vertices and edges $e_1,e_2,e_3$. Let $v:=e_1+e_2+e_3\in\mathbb{R}^{E(\Gamma)}$. The fan $\Sigma_{\Gamma}$ is described in Example \ref{ex:fans}. In Figure \ref{fig:circular'} we depict a section of this fan. The fan $\Sigma_Y$ is the quotient of $\Sigma_\Gamma$ by $v$ and is described in Figure \ref{fig:p2b},  where $\overline{e_i}$ is the class of $e_i$ modulo $v$. We note that $\overline{e_1}+\overline{e_2}+\overline{e_3}=0$, which implies that $\Sigma_Y$  is the fan of $\mathbb{P}^2$ blown up at 3 points.
\begin{figure}[htb]

\begin{minipage}{0.48\linewidth}
\[
\begin{xy} <10pt,0pt>:
(-4,0)*{\scriptstyle\bullet}="a"; 
(4,0)*{\scriptstyle\bullet}="b";
(0,6)*{\scriptstyle\bullet}="c";
(0,2)*{\scriptstyle\bullet}="v";
(-2,3)*{\scriptstyle\bullet}="d";
(2,3)*{\scriptstyle\bullet}="e";
(0,0)*{\scriptstyle\bullet}="f";
{\ar@{-}"a"*{};"b"*{}};
{\ar@{-}"a"*{};"c"*{}};
{\ar@{-}"a"*{};"v"*{}};
{\ar@{-}"c"*{};"b"*{}};
{\ar@{-}"v"*{};"b"*{}};
{\ar@{-}"v"*{};"c"*{}};
{\ar@{-}"v"*{};"d"*{}};
{\ar@{-}"v"*{};"e"*{}};
{\ar@{-}"v"*{};"f"*{}};
"v"+(-0.25,-0.55)*{v};
"a"+(-0.45,-0.25)*{e_2};
"b"+(0.6,-0.25)*{e_1};
"c"+(0.25,0.35)*{e_3};
\end{xy}
\]
\caption{A section of the fan $\Sigma_\Gamma$.}
\label{fig:circular'}
\end{minipage}\quad
\begin{minipage}{0.48\linewidth}
\[
\begin{xy} <10pt,0pt>:
(-4,-4)="a"; 
(4,0)="b";
(0,4)="c";
(0,0)="v";
(4,4)="d";
(-4,0)="e";
(0,-4)="f";
{\ar@{->}"v"*{};"a"*{}};
{\ar@{->}"v"*{};"b"*{}};
{\ar@{->}"v"*{};"c"*{}};
{\ar@{->}"v"*{};"d"*{}};
{\ar@{->}"v"*{};"e"*{}};
{\ar@{->}"v"*{};"f"*{}};
"a"+(-0.35,0.6)*{\overline{e_2}};
"b"+(0,-0.5)*{\overline{e_1}};
"c"+(0.25,0.35)*{\overline{e_3}};
"d"+(0.25,0.25)*{\overline{e_1}+\overline{e_3}};
"e"+(0.25,0.5)*{\overline{e_2}+\overline{e_3}};
"f"+(1.65,0.25)*{\overline{e_1}+\overline{e_2}};
\end{xy}
\]
\caption{The fan $\Sigma_Y$.}
\label{fig:p2b}
\end{minipage}
\end{figure}
\end{Exa}

Recall that a fan $\Sigma$ in $\mathbb{R}^n$ is complete if its support is all of $\mathbb{R}^n$.
\begin{Prop}
The fan $\Sigma_Y$ is complete and $\E_Y$ is proper.
\end{Prop}
\begin{proof}
The fact that $\E_Y$ is proper follows from the fact that $\Sigma_Y$ is complete and \cite[Theorem 3.4.11]{CLS}. Let us check that $\Sigma_Y$ is complete.\par 
 Let $v:=\sum_{e\in E(\Gamma)}e=\sum_{i=1}^m v_i$. Since $v\in U_\Gamma$, we have that all equivalence classes in $\mathbb{R}^{E(\Gamma)}/U_\Gamma$ have a representative in $\mathbb{R}_{>0}^{E(\Gamma)}$ (just choose a representative of this class and sum up a suitable multiple of $v$). Then, we have that $q_\Gamma({\mathbb{R}_{>0}^{E(\Gamma)}})=\mathbb{R}^{E(\Gamma)}/U_\Gamma$.  Since the support of $\Sigma_\Gamma$ is $\mathbb{R}_{\geq0}^{E(\Gamma)}$ we have that the support of $\Sigma_Y$ is $\mathbb{R}^{E(\Gamma)}/U_\Gamma$, hence $\Sigma_Y$ is complete.
\end{proof}
   For all $\Gamma$ we have $\Sigma_\Gamma=\prod \Sigma_{\Gamma_i}$, where $\Gamma_i$ are the biconnected components of $\Gamma$. Then 
\begin{equation}
\label{eq:prod}
\E_Y=\prod \E_{Y_i}
\end{equation}
where $Y_i$ are the biconnected components of $Y$, i.e., the subcurve $Y_i\subset Y$ corresponds to the subgraph $\Gamma_i$ of $\Gamma$.

We denote by $\mathcal{B}$ the set of all bonds of $\Gamma$ and $\mathbb{P}(\mathbb{C}^B)$ the projectivized of the vector space $\mathbb{C}^B$ with homogeneous coordinates $(x_e^B)_{e\in B}$. In the following theorem, we will relate $\E_Y$ with the projective spaces $\mathbb{P}(\mathbb{C}^B)$. But first we start with the following proposition.

\begin{Prop}
\label{prop:map}
Given a bond $B$ of the dual graph $\Gamma$ of $Y$ there exists a natural surjective toric map $\rho_B\colon \E_Y\to \mathbb{P}(\mathbb{C}^B)$. 
\end{Prop}
\begin{proof}
By Equation \eqref{eq:prod} above, and the fact that each bond of $\Gamma$ is a bond of one of its biconnected components we can assume that $\Gamma$ is biconnected. \par
   To prove the statement it is enough to give a natural surjective integral map 
\[
\overline{r_B}\colon\frac{\mathbb{R}^{E(\Gamma)}}{U_\Gamma}\lra \frac{\mathbb{R}^B}{U_B},
\] 
where $U_B$ is the integral linear subspace generated by $v:=\sum_{e\in B}e$, such that the image of each cone in the fan $\Sigma_Y$ is contained in a cone of the fan $\Sigma_{\mathbb{P}(\mathbb{C}^B)}$, in other words the map $\overline{r_B}$ is compatible with the fans $\Sigma_Y$ and $\Sigma_{\mathbb{P}(\mathbb{C}^B)}$. \par
	For each $T\subset B$ define the cone 
\[
\sigma_T:=\{(x_e)\;|\; 0\leq x_e\leq x_{e'} \;\text{for every $e\in T$ and $e'\in B$}\}\subset\mathbb{R}^B 
\]
where $(x_e)$ are the coordinates of $\mathbb{R}^B$. Then define the fan 
\[
\Sigma_B:=\{\sigma_T\;|\; T\subset B\}.
\]
Clearly the image of the fan $\Sigma_B$ via the quotient map $q_B\col\mathbb{R}^B\to\mathbb{R}^B/U_B$ is the fan $\Sigma_{\mathbb{P}(\mathbb{C}^B)}$. \par
 The projection map $r_B\colon\mathbb{R}^{E(\Gamma)}\to\mathbb{R}^B$ is such that the image of $U_\Gamma$ is $U_B$, therefore we have the commutative diagram defining $\overline{r_B}$
\[
\SelectTips{cm}{11}
\begin{xy}<16pt,0pt>:
\xymatrix{\mathbb{R}^{E(\Gamma)} \ar[d]^{q_\Gamma}\ar[r]^{r_B} & \mathbb{R}^B\ar[d]^{q_B}\\
                 \displaystyle\frac{\mathbb{R}^{E(\Gamma)}}{U_\Gamma} \ar[r]^{\overline{r_B}} &\displaystyle \frac{\mathbb{R}^B}{U_B} 
								}
	\end{xy}
\]
Since the quotient map $q_\Gamma$ is compatible with the fans $\Sigma_\Gamma$ and $\Sigma_Y$ and the map $q_B$ is compatible with the fans $\Sigma_B$ and $\Sigma_{\mathbb{P}(\mathbb{C}^B)}$, to prove that $\overline{r_B}$ is compatible with the fans $\Sigma_Y$ and $\Sigma_{\mathbb{P}(\mathbb{C}^B)}$, it is enough to prove that $r_B$ is compatible with the fans $\Sigma_\Gamma$ and $\Sigma_B$. This follows from Proposition \ref{prop:bond}. Indeed if $K(\Gamma,p)$ is a cone in $\Sigma_\Gamma$ then there exists a lower set $T\subset B$ such that $e\<_pe'$ for every $e\in T$ and $e'\in B$, hence the image $r_B(K(\Gamma,p))$ is contained in $\sigma_T$.
\end{proof}

 Recall the definition of sum of two cuts given in Section \ref{subsec:graphs}.
\begin{Thm} The map 
\label{thm:maininclusion}
\[
\rho\colon \E_Y\xrightarrow{(\rho_B)_{B\in\mathcal{B}}}\prod_{B\in\mathcal{B}}\mathbb{P}(\mathbb{C}^B)
\]
is a closed immersion and its image has equations
\[
\begin{array}{cl}
x_{e_1}^{B_1}x_{e_2}^{B_2}=x_{e_2}^{B_1}x_{e_1}^{B_2}&\text{for every $B_1, B_2\in\mathcal{B}$  and  $e_1,e_2\in B_1\cap B_2$}\\
&\\
x_{e_1}^{B_1}x_{e_2}^{B_2}x_{e_3}^{B_3}=x_{e_2}^{B_1}x_{e_3}^{B_2}x_{e_1}^{B_3}& \begin{minipage}{0.6\linewidth}for every $B_1,B_2,B_3\in\mathcal{B}$ such that $B_1+B_2=B_3$ and $e_1\in B_1\cap B_3$, $e_2\in B_1\cap B_2$, $e_3\in B_2\cap B_3$.\end{minipage}
\end{array}
\]

\end{Thm}

\begin{proof}
Keep the notation of the proof of Proposition \ref{prop:map}.\par
We can assume that $\Gamma$ is biconnected and has at least two edges (the result is clear if $\Gamma$ has only one edge, or no edges). The map $\rho$ is induced by the integral linear map 
\[
\overline{r}\colon\frac{\mathbb{R}^{E(\Gamma)}}{U_\Gamma}\xrightarrow{\oplus\overline{r_B}} \bigoplus_{B\in\B}\frac{\mathbb{R}^B}{U_B}.
\]
We claim that $\overline{r}$ is injective. Indeed, if $(x_e)_{e\in E(\Gamma)}$ is an element of $\mathbb{R}^{E(\Gamma)}$ such that $ r_B((x_e)_{e\in E(\Gamma)})\in U_B$ for all $B\in\mathcal{B}$, then
$x_e=x_{e'}$ for every $e,e' \in B$ for every $B\in\mathcal{B}$. We note that, since $\Gamma$ is biconnected, the cut $B_v:=E(\{v\},\{v\}^c)$ is a bond, for every vertex $v\in V(\Gamma)$. Moreover, since $\Gamma$ is connected, there exists a sequence of (possibly not all different) vertices $(v_i)_{i=1,\ldots,m}$ such that $E(\Gamma)=\{v_i| i=1,\ldots,m\}$ and such that $v_i$ and $v_{i+1}$ are connected by an edge $e_i\in E(\Gamma)$. We have that $\{e_{i-1},e_i\}\subset B_{v_i}$ for $i=2,\ldots,m$, then $x_{e_{i-1}}=x_{e_i}$ for every $i=2,\ldots,m$. However every edge $e\in E(\Gamma)$ is connected to a vertex $v_j$ for some $j=1,\ldots,m$, then $x_e=x_{e_j}$ because $e,e_j\in B_{v_j}$, and hence $x_e=x_{e'}$ for every $e,e'\in E(\Gamma)$, which implies that $(x_e)_{e\in E(\Gamma)}\in U_\Gamma$. This proves the claim.

 By Proposition \ref{prop:bond1}, given a collection $\T=(T_B)_{B\in\B}$ of nonempty subsets $T_B\subset B$, the preimage via $\overline{r}$ of the cone 
\[
\prod_{B\in \B} q_B(\sigma_{T_B})
\]
lies inside a single cone of $\Sigma_Y$. Hence the map $\rho$ is a closed immersion.\par
 Every element of $\oplus(\mathbb{Z}^{B}/U_B)^\vee$ corresponds to a rational function of $\prod\mathbb{P}(\mathbb{C}^B)$, hence, to find the equations of $\rho(\E_Y)$ it is enough to find generators for the kernel of the map
\[
\overline{r}^\vee\colon M:=\bigoplus_{B\in\mathcal B}\left(\frac{\mathbb{Z}^B}{U_B}\right)^\vee\lra\left(\frac{\mathbb{Z}^{E(\Gamma)}}{U_\Gamma}\right)^\vee.
\]
The elements of $(\mathbb{Z}^B/U_B)^\vee$ are the elements of $(\mathbb{Z}^B)^\vee$ that vanish on $U_B$. Hence these elements can be written as 
\[
\sum_{e\in B} \ell_e e^\vee
\]
 with $\ell_e\in\mathbb{Z}$ and $\sum \ell_e=0$. Moreover an element of the form $\sum \ell_e e^\vee$ corresponds to the rational function $\prod x_e^{\ell_e}$ on $\mathbb{P}(\mathbb{C}^B)$. Since $(e_1^\vee-e_2^\vee)\in (\mathbb{Z}^B/U_B)^\vee$ corresponds to the rational function $x_{e_1}^B/x_{e_2}^B$ on $\mathbb{P}(\mathbb{C}^B)$, the result follows from Lemma \ref{lem:kernel} below.
\end{proof}
	
	Before stating the next lemma, let us fix some notation. As in the proof of Theorem \ref{thm:maininclusion}, define
\[
M:=\bigoplus_{B\in\mathcal B}\left(\frac{\mathbb{Z}^B}{U_B}\right)^\vee.
\]
An element in $M$ is of the form 
\[
\bigoplus_{B\in\B}\left(\sum_{e\in B}\ell_{e,B}e^\vee\right)_B.
\]	
To avoid cumbersome notation we will always not write the summands which are zero. For instance the notation $(e_1^\vee+e_2^\vee)_{B_1}\oplus (e_3^\vee+e_4^\vee)_{B_2}$ means the element
\[
\bigoplus_{B\in\B\backslash\{B_1,B_2\}}(0)_B\oplus(e_1^\vee+e_2^\vee)_{B_1}\oplus (e_3^\vee+e_4^\vee)_{B_2}.
\]

\begin{Lem}
\label{lem:kernel}
If $\Gamma$ is biconnected, then the kernel of the map $\overline{r}^\vee$ is generated by elements of the form 
\begin{equation}
\label{eq:gen}
(e_1^\vee-e_2^\vee)_{B_1}\oplus (e_2^\vee-e_1^\vee)_{B_2}\quad\text{and}\quad
(e_1^\vee-e_2^\vee)_{B_1}\oplus (e_2^\vee-e_3^\vee)_{B_2} \oplus (e_3^\vee-e_1^\vee)_{B_3}
\end{equation}
where on left hand side we have $e_1,e_2\in B_1\cap B_2$ and on the right hand side we have the equality $B_3=B_1+B_2$ and $e_1\in B_1\cap B_3$, $e_2\in B_1\cap B_2$, $e_3\in B_2\cap B_3$.
\end{Lem}	

\begin{proof}
Let $W$ be the subspace generated by the elements in Equation \eqref{eq:gen}. We call the elements on the left (respectively right) hand side in Equation \eqref{eq:gen} of type 1 (respectively of type 2). Clearly $W$ is contained in the kernel of the map $\overline{r}^\vee$. Conversely let
\[
\bigoplus_{B\in\B}\left(\sum_{e\in B}\ell_{e,B}e^\vee\right)_B
\]
be an element of the kernel. This means that 
\begin{equation}
\label{eq:0}
\sum_{B\in \B}\ell_{e,B}=0\quad\text{for every $e\in E(\Gamma)$.}
\end{equation}\par
First, we claim that every element in $M$ is equivalent modulo $W$ to an element of the form
\begin{equation}
\label{eq:v}
\bigoplus_{v\in V(\Gamma)}\left(\sum_{e\in B_v}\ell_{e,v}e^\vee\right)_{B_v}
\end{equation}
where $B_v:=E(\{v\},\{v\}^c)$, which is a bond because $\Gamma$ is biconnected. To prove the claim, note that the generators of $M$ are of the form $(e_1^\vee-e_2^\vee)_B$ for every $B\in\B$ and $e_1,e_2\in B$. Fix $B=E(V,V^c)$. If $e_1$ and $e_2$ are incident to the same vertex $v$ in $V$, then, using generators of type 1 of $W$, we have
\[
(e_1^\vee-e_2^\vee)_B=(e_1^\vee-e_2^\vee)_{B_v} \; (\text{mod}\; W).
\]
Otherwise, if $e_i$ is incident to a vertex $v_i\in V$, then, since $\Gamma(V)$ is connected, there is a partition of $V=V_1\cup V_2$, with $v_i\in V_i$  such that  $\Gamma(V_i)$ is connected (see Lemma \ref{lem:partition}). Moreover, since $\Gamma(V^c)$ is also connected, so are $\Gamma(V_1^c)$ and $\Gamma(V_2^c)$. Let $e_3$ be an edge in $E(V_1,V_2)$ and define $B_i:=E(V_i,V_i^c)$. Then, using a generator of type 2 of $W$, we have
\[
(e_1^\vee-e_2^\vee)_B=(e_1^\vee-e_3^\vee)_{B_1}\oplus(e_3^\vee-e_2^\vee)_{B_2}\; (\text{mod}\; W).
\]
Since $|V_i|<|V|$, iterating the argument we get the claim.\par

Now, since every edge is incident at exactly two vertices, we can write every element in Equation \eqref{eq:v} which satisfies Equation \eqref{eq:0} as a linear combination of elements of the form
\[
(e_1^\vee-e_2^\vee)_{B_{v_1}}\oplus(e_2^\vee-e_3^\vee)_{B_{v_2}}\oplus\cdots\oplus(e_k^\vee-e_1^\vee)_{B_{v_k}}
\]
for some cycle whose edges are $e_i$ and vertices $v_i$. Now, by Lemma \ref{lem:partition}, there is a partition $V(\Gamma)=V_1\cup\cdots\cup V_k$ such that $v_i\in V_i$ and $\Gamma(V_i)$ is connected. Clearly, the graph $\Gamma(V_i^c)$ is also connected for every $i$. Define $B_i:=E(V_i,V_i^c)$. Then, using generators of type 1, we have
\[
(e_i^\vee-e_{i+1}^\vee)_{B_{v_i}}=(e_i^\vee-e_{i+1}^\vee)_{B_i}\;(\text{mod}\; W).
\]
So, we are reduced to prove that elements of the form
\[
(e_1^\vee-e_2^\vee)_{B_1}\oplus(e_2^\vee-e_3^\vee)_{B_2}\oplus\cdots\oplus(e_k^\vee-e_1^\vee)_{B_k}
\]
are in $W$.\par
Using generators of type 2, we get
\[
(e_{k-1}^\vee-e_k^\vee)_{B_{k-1}}\oplus(e_k^\vee-e_1^\vee)_{B_k}=(e_{k-1}^\vee-e_1^\vee)_{B_{k-1}'}\;(\text{mod}\; W)
\]
where $B_{k-1}'=E(V_{k-1}\cup V_k,V_{k-1}^c\cap V_k^c)$, in other words $B_{k-1}'=B_{k-1}+B_k$. Since $V_1,V_2,\ldots, V_{k-2}, V_{k-1}'$ is still a partition of $V(\Gamma)$ we can iterate the argument to reduce to the case where $k=3$. In this case the element
\[
(e_1^\vee-e_2^\vee)_{B_1}\oplus(e_2^\vee-e_3^\vee)_{B_2}\oplus(e_3^\vee-e_1^\vee)_{B_3}
\]
is one of the generators of type 2, because $V_1\cup V_2=V_3^c$. This finishes the proof of the lemma.

Since, in the case where $\Gamma$ is biconnected, $\E_Y$ and $\mathbb{P}(\mathbb{C}^{E(\Gamma)})$ are birational we have the following corollary.

\begin{Cor}
\label{cor:P}
The rational map 
\[
\mathbb{P}(\mathbb{C}^{E(\Gamma)})\dashrightarrow \prod_{B\in\mathcal{B}}\mathbb{P}(\mathbb{C}^B)
\]
induced by the quotients $\mathbb{R}^{E(\Gamma)}\to\mathbb{R}^B$ and the map $\rho\col \E_Y\to\prod_{B\in\mathcal{B}}\mathbb{P}(\mathbb{C}^B)$ have the same images.
\end{Cor}
\end{proof}

We point out that Li in \cite{Li} gave a functorial description of such images and computed their Hilbert polynomials.\par

Finally, as a consequence of Proposition \ref{prop:star}, we can give a more geometric construction of $\E_Y$ which resembles Kapranov's construction of $\overline{\M}_{0,n+3}$, see \cite{K}.

\begin{Thm}
\label{thm:mainEX}
Let $Y$ be a biconnected stable curve with dual graph $\Gamma$. Set $n:=|E(\Gamma)|-1$. Then, there exists a birational map $\E_Y\to\mathbb{P}(\mathbb{C}^{E(\Gamma)})=\mathbb{P}^n$ which is the composition of the following blowups of $\mathbb{P}^n$:
\begin{enumerate}
\item[(1)] first, blowup the points given by equations $x_e=0$ for every $e\in E(\Gamma/\{e_1\})$ such that $e_1$ runs through all elements in $E(\Gamma)$ for which $\Gamma/\{e_1\}$ is biconnected;
\item[(2)] then, blowup the lines given by equations $x_e=0$ for $e\in E(\Gamma/\{e_1,e_2\})$ such that $\{e_1,e_2\}$ runs through all subsets of cardinality two of $E(\Gamma)$ for which $\Gamma/\{e_1,e_2\}$ is biconnected;
\[
\cdots
\]
\item[(n-1)] finally, blowup the $(n-2)$-dimensional linear subspaces given by equations $x_e=0$ for $e\in E(\Gamma/S)$ such that $S$ runs through all subsets of cardinality $n$ of $E(\Gamma)$ for which $\Gamma/S$ is biconnected.
\end{enumerate}

\end{Thm}
\begin{proof}
   The theorem follows from Proposition \ref{prop:star}. Indeed, up to taking the quotient by $U_\Gamma$ as in Equation \eqref{eq:Vgamma}, the sequence of blowups performed in the stated theorem corresponds to the sequence of star subdivisions of the fan
\[
\Sigma_{0,\Gamma}=\{\text{cone}(A)|A\subset E(\Gamma)\}
\]
as in Proposition \ref{prop:star}. 
\end{proof}

\begin{Cor}
\label{cor:maino}
Let $Y$ be a stable curve and $\V_Y$ its space of versal deformations. If $\pi\colon B\to \V_Y$ is the Main\`o blowup described in Section \ref{sec:maino}, then $\E_Y$ is the fiber $\pi^{-1}(0)$ of the special point $0$ of $\V_Y$.
\end{Cor}
\begin{proof}
The relevant locus $\R_i$ in $\V_Y$ corresponds precisely to the specializations $\Gamma\leadsto \Gamma'$ such that $\Gamma'$ is biconnected and $|E(\Gamma')|=i$.
\end{proof}

\section{Final remarks}
\label{sec:final}
  In this section we collect some observations that can be useful to get a modular description of $\E_Y$, which was the original motivation of the work and yet unclear to us. We had some small progress in this direction, but due to the already large length of the paper, we decided to just summarize it here.\par
	
	At a first glance, it is natural to expect that the variety $\E_Y$ should be isomorphic to some Hassett's moduli space of vertex weighted rational curves, constructed in \cite{H}. This works for some examples. Indeed, if $Y$ is a circular curve with $\delta$ nodes then $\E_Y$ is isomorphic to the Losev-Manin moduli space $L_{\delta}$, which is the Hassett moduli spaces $\M_{0,\A}$ for $\A=(1,1, 1/\delta,\ldots, 1/\delta)$. Also, if $Y$ is a curve with only two components and $\delta$ nodes then $\E_Y=\mathbb{P}^{\delta-1}$, which is also a Hassett moduli space. However, we later found curves for which $\E_Y$ is not isomorphic to a Hassett moduli space. \par
 Then, to fully describe $\E_Y$ as a moduli space, we suggest a small generalization of the moduli problem considered by Hassett, as follows.\par
 Let $\widetilde\A=(\underline{a}_1,\ldots, \underline{a}_n)$ be a collection of $n$ vectors $\underline{a}_i=(a_{i,j})\in \mathbb{Q}^m$ with $0\leq a_{i,j}\leq 1$. A family of nodal curves of genus $g$ with marked points $(C,s_1,\ldots,s_n)\xrightarrow{\pi} B$ is stable of type $\widetilde\A$ if
\begin{enumerate}
\item the sections $s_1,\ldots, s_n$ lie in the smooth locus of $\pi$, and for any subset $\{s_{i_1}\ldots, s_{i_r}\}$ with nonempty intersection we have $a_{i_1,j}+\ldots+a_{i_r,j}\leq 1$ for every $j=1,\ldots,m$;
\item $K_\pi+a_{1,j}s_1+\ldots+a_{n,j}s_n$ is $\pi$-relatively ample for every $j=1,\ldots,m$.
\end{enumerate}
We expect that the construction done by Hassett should also work in the above setting and give rise to a moduli space $\overline{\M}_{g,\widetilde\A}$ parametrizing these data.\par
  Now, for each nodal curve $Y$ one can construct a collection $\widetilde\A_Y$ as follows. Assume that $Y$ has $\delta$ nodes and let $\Gamma$ be its dual graph. Denote by $E(\Gamma)=\{e_1,\ldots,e_\delta\}$ its set of edges . Moreover, let $\B$ be the set of bonds of $\Gamma$. Define $\widetilde\A_Y:=(\underline{a}_1,\ldots,\underline{a}_{\delta},\underline{b}_0,\underline{b}_{\infty})$ where $\underline{a}_i=(a_{i,B})_{B\in\B}$ with
\[
a_{i,B}=\begin{cases}
         0 &\text{if } e_i\notin B\\
				 \frac{1}{|B|}&\text{if } e_i\in B,
				\end{cases}
\]				
$\underline{b}_0=(1,1,\ldots,1)$ and $\underline{b}_{\infty}=(b_{i,B})_{B\in \B}$ with $b_{i,B}=1/|B|$.\par

\begin{Conj}
$\E_Y$ is isomorphic to $\overline{\M}_{0,\widetilde\A_Y}$.
\end{Conj}

 We note that for each $B\in \B$ if we set $\A_B$ to be the subsequence of 
\[
(a_{1,B} , a_{2,B} , \ldots , a_{\delta,B} , b_{0,B} , b_{\infty,B})
\]
containing all the positive weights, then we should have reduction morphisms $\overline{\M}_{0,\widetilde\A_Y}\to\overline{\M}_{0,\A_B}=\mathbb{P}(\mathbb{C}^B)$ for each $B\in \B$, and these morphisms should recover Proposition \ref{prop:map} and Theorem \ref{thm:maininclusion}.\\

Finally, we give an idea of how to construct a ``universal'' family over $\E_Y$ containing the one parametrizing enriched structures on $Y$.

Let $Y$ be a nodal curve and $Z$ be a connected subcurve of $Y$ with connected complement $Z^c$, and set $\Delta_Z:=Z\cap Z^c$. Notice that $Z$ defines a bond $B_Z$ of the dual graph $\Gamma$ of $X$. Let $\L_Z$ be the rank-$1$ torsion-free simple sheaf over $\mathbb{P}(\mathbb{C}^{B_Z})\times Y$ constructed considering the natural identification $\mathbb{P}(\mathbb{C}^{B_Z})\to\mathbb{P}(\text{Hom}(\O_{\Delta_Z},\O_{\Delta_Z}))$ and using the coordinates $(x_e)\in\mathbb{P}(\mathbb{C}^{B_Z})$ to glue the two sheaves $\O_Z(\Delta_Z)$ and $\O_{Z^c}(-\Delta_Z)$ along $\Delta_Z$. We refer to \cite[Section 4.1.1]{Rizzo} for a similar construction.\par
  Via Proposition \ref{prop:map} we can pull these sheaves back to get a collection of rank-$1$ torsion-free simple sheaves $\underline{\L}:=(\L_Z)$  over $\E_Y\times Y$, where $Z$ varies over all connected subcurves with connected complement of $Y$. Moreover, if $T$ is the maximal torus in $\E_Y$, then the restriction of $\L_Z$ to $T\times Y$ is locally free. By the equations in Theorem \ref{thm:maininclusion}, we have
\[
\L_{Z_1}\otimes\L_{Z_2}=\L_{Z_1\cup Z_2}
\]
if $Z_1$ and $Z_2$ have no common components and $Z_1\cup Z_2$ is connected with connected complement, and
\[
\L_{Z^c}=\L_{Z}^{-1}.
\]
 Thus, we can regard this collection $\underline{\L}$ restricted to $T\times Y$ as the universal family of enriched structures in the sense of Main\`o. The modular description of $\E_Y$ should follow from a geometric description of the collection $\underline{\L}$, which turns out to be the universal family over $\E_Y$. So, the object we need to get such a modular description should be rank-1 torsion-free simple sheaves over $Y$ with certain properties induced by the equations in Theorem \ref{thm:maininclusion}. However we were not able to translate these equations into explicit conditions.

\section*{Acknowledgments}

  We would like to thank Margarida Melo and Filippo Viviani for precious discussions and advices during the early development of this paper. We also thank Lucia Caporaso for helpful remarks on a preliminary version of the paper. A special thanks goes to Margarida for her lectures on tropical curves taught at Universidade Federal Fluminense in 2013. \par
	We thank the referees for carefully reading the paper and for the constructive suggestions.

\bigskip
\noindent{\smallsc Alex Abreu and Marco Pacini, Universidade Federal Fluminense,\\ Rua M. S. Braga, s/n, Valonguinho, 24020-005 Niter\'oi (RJ) Brazil.}\\
{\smallsl E-mail addresses: \small\verb?alexbra1@gmail.com?  and \small\verb?pacini.uff@gmail.com?}


\begin{thebibliography}{llll}




\bibitem{ACP}  D. Abramovich, L. Caporaso and S. Payne,
\emph{The tropicalization of the moduli space of curves.}
Available at arXiv:1212.0373 (2012).

\bibitem{AC} O. Amini and L. Caporaso,
\emph{Riemman-Roch theory for weighted graphs and tropical curves.}
Ad. Math. {\bf 240} (2013) 1--23.

\bibitem{ACG} E. Arbarello, M. Cornalba and P. A. Griffiths,
\emph{Geometry of algebraic curves, Volume II}.
With a contribution by J. Harris. 
Grundlehren der Mathematischen Wissenschaften [Fundamental Principles of Mathematical Sciences], {\bf 268}. Springer, Heidelberg, 2011.

\bibitem{BN} M. Baker and S. Norine,
\emph{Riemann-Roch and Abel-Jacobi theory on a finite graph}.
Adv. Math. {\bf 215 (2)} (2007) 766--788.

\bibitem{BM} V. Batyrev and M. Blume,
\emph{The function of toric varieties associated with Weyl chamber and Losev-Manin moduli spaces.}
Tohoku Math. J. {\bf 63 (2)} (2011), no. 4, 581–-604. 

\bibitem{BMV} S. Brannetti, M. Melo and F. Viviani.
\emph{On the tropical Torelli Map},
Adv. Math. {\bf 226 (3)} (2011) 2546--2586.

\bibitem{Caporaso} L. Caporaso, 
\emph{Algebraic and tropical curves: comparing their moduli spaces}.
In: Handbook of Moduli, Volume I. G. Farkas, I. Morrison (Eds.), Advanced Lectures in Mathematics, Volume XXIV (2012), 119--160.

\bibitem{Caporaso1} L. Caporaso,
\emph{Geometry of tropical moduli spaces and linkage of graphs}.
 Journal of Combinatorial Theory, Series A. 119 (2012) 579--598
	
\bibitem{CV} L. Caporaso and F. Viviani, 
\emph{Torelli theorem for graphs and tropical curves}.
Duke Math. Journal {\bfseries 153} (2010) 129-–171.

\bibitem{CHMR} R. Cavalieri, S. Hampe, H. Markwig and D. Ranganathan,
\emph{Moduli spaces of rational weighed stable curves and tropical geometry}.
Available at arXiv:1404.7426 (2014).

\bibitem{CMR} R. Cavalieri, H. Markwig and D. Ranganathan,
\emph{Tropicalizing the space of admissible covers}.
Available at arXiv:1401.4626v1  (2014).

\bibitem{CDPR} F. Cools, J. Draisma, S. Payne, E. Robeva,
\emph{A tropical proof of the Brill-Noehter Theorem}
Adv. Math. {\bf 230} (2012) 759--776.

\bibitem{CLRS} T. Cormen, C. Leiserson, R. Rivest, C. Stein,
\emph{Introduction to algorithms}.
Third edition. MIT Press, Cambridge, MA, (2009).

\bibitem{CLS} D. Cox, J. Little and H. Schenck,
\emph{Toric Varieties}.
Graduate Studies in Mathematics, Volume 124 (2011).


\bibitem{EH} D. Eisenbud and J. Harris, 
\emph{Limit linear series: Basic theory}. 
Invent. Math. {\bf 85} (1986) 337--371.

\bibitem{Es} E. Esteves,
\emph{Linear systems and ramification points on 
reducible nodal curves}, 
In: Algebra Meeting (Rio de Janeiro, 1996), 21--35, 
Mat. Contemp., vol. 14, 1998, Soc. Bras. Mat., 
Rio de Janeiro, 1998.


\bibitem{EM} E. Esteves and N. Medeiros, \emph{Limit canonical systems on 
curves with two components}. Invent. Math. {\bf 149} (2002) 267--338.

\bibitem{H}  B. Hasset,
\emph{Moduli space of weighted pointed stable curves}.
Adv. Math. {\bf 173} (2003), no.2, 316--352.

\bibitem{K} M. Kapranov,
\emph{Veronese curves and Grothendieck-Knudsen moduli space $\overline{M}_{0,n}$.}
J. Algebraic Geom. {\bf 2} (1993), no. 2, 239-–262.

\bibitem{Li} B. Li, 
\emph{Images of rational maps of projective spaces.}
Available at arXiv:1310.8453 (2014). 

\bibitem{maino} L. Main\`o,
\emph{Moduli space of enriched stable curves}. Ph.D.~Thesis, 
Harvard University, Cambridge, 1998.

\bibitem{M} G. Mikhalkin,
\emph{Moduli spaces of rational tropical curves}, in Proceedings of G\"okova GEometry-Topology Conference 2006, G\"okova Geometry/Topology Conference (GGT), G\"okova, 2007, 39--51.

\bibitem{MZ} G. Mikhalkin and I. Zharkov,
\emph{Tropical curves, their Jacobians and Theta functions}, in Proceedings of the International Conference on Curves and Abelian Varieties in Honor of Roy Smith's 65th Birthday, in: Contemp. Math., vol. 465, 2007, 203--231.

\bibitem{O} B. Osserman, \emph{A limit linear series moduli scheme},
Annales de l'Institut Fourier {\bf 56} (2006) 1165--1205.

\bibitem{O1} B. Osserman, \emph{Limit linear series for curves not of compact type},
Available at arXiv:1406.6699 (2014).

\bibitem{P} S. Payne,
\emph{Analytification is the limit of all tropicalizations},
 Math. Res. Lett. {\bf 16} (2009), no. 3, 543–-556. 

\bibitem{Rizzo} P. Rizzo,
\emph{Level-$\delta$ and stable limit linear series on singular curves},
Tese de doutorado, IMPA, 2013.

\bibitem{U} M. Ulirsch,
\emph{Tropical geometry of moduli spaces of weighted stable curves}.
Available at arXiv:1405.6940 (2014).


\end{thebibliography}
\end{document}